 \documentclass[reqno,11pt]{amsart}
 \usepackage{amsmath, amsthm, a4, latexsym, amssymb}
\usepackage[unicode]{hyperref}
\usepackage{srcltx}

\def\@rmrk#1#2{\refstepcounter
    {#1}\@ifnextchar[{\@yrmrk{#1}{#2}}{\@xrmrk{#1}{#2}}}

\makeatletter\@addtoreset{equation}{section}\makeatother

 \newfont{\bfit}{cmbxti10 scaled 1200}

\renewcommand{\d}{{\rm d}}

 \newcommand{\eps}{\varepsilon}

 \newcommand{\R}{\mathbb{R}}
 \newcommand{\N}{\mathbb{N}}

 \newcommand{\E}{\mathbb{E}}
 \renewcommand{\P}{\mathbb{P}}
 \def\1{{\mathchoice {1\mskip-4mu\mathrm l} 
{1\mskip-4mu\mathrm l}
{1\mskip-4.5mu\mathrm l} {1\mskip-5mu\mathrm l}}}

 \newcommand{\Mcal}{{\mathcal M}}

 \newcommand{\skrix}{{\mathcal X}}
 
\newcommand{\weak}{{\Rightarrow}}
\newcommand{\vague}{{\hookrightarrow}}

\newcommand{\heap}[2]{\genfrac{}{}{0pt}{}{#1}{#2}}

\newcommand{\ssup}[1] {{\scriptscriptstyle{({#1}})}}

\renewcommand{\subsection}{\secdef \subsct\sbsect}
\newcommand{\subsct}[2][default]{\refstepcounter{subsection}
\vspace{0.15cm}
{\flushleft\bf \arabic{section}.\arabic{subsection}~\bf #1  }
\nopagebreak\nopagebreak}
\newcommand{\sbsect}[1]{\vspace{0.1cm}\noindent
{\bf #1}\vspace{0.1cm}}

{\nopagebreak {\hfill\rule{2mm}{2mm}}\\ }

\newtheorem{theorem}{Theorem}[section]
\newtheorem{lemma}[theorem]{Lemma}
\newtheorem{cor}[theorem]{Corollary}
\newtheorem{prop}[theorem]{Proposition}

\newtheoremstyle{thm}{1.5ex}{1.5ex}{\itshape\rmfamily}{}
{\bfseries\rmfamily}{}{2ex}{}

\newtheoremstyle{rem}{1.3ex}{1.3ex}{\rmfamily}{}
{\itshape\rmfamily}{}{1.5ex}{}
\theoremstyle{rem}
\newtheorem{remark}{{\slshape\sffamily Remark}}[]

\refstepcounter{subsubsection}

\def\thebibliography#1{\section*{References}
  \list%
  {\arabic{enumi}.}
    {\settowidth\labelwidth{[#1]}\leftmargin\labelwidth
    \advance\leftmargin\labelsep
    \parsep0pt\itemsep0pt
    \usecounter{enumi}}
    \def\newblock{\hskip .11em plus .33em minus .07em}
    \sloppy                   
    \sfcode`\.=1000\relax}

\begin{document}
\title[Brownian occupation measures, compactness and large deviations]
{\large Brownian occupation measures, compactness and large deviations}
\author[Chiranjib Mukherjee and S.R.S. Varadhan]{}
\maketitle
\thispagestyle{empty}
\vspace{-0.5cm}

\centerline{\sc Chiranjib Mukherjee\footnote{Courant Institute of Mathematical Sciences and WIAS Berlin, 251 Mercer Street, New York 10012, {\tt mukherjee@cims.nyu.edu}},
S. R. S. Varadhan, \footnote{Courant Institute of Mathematical Sciences, 251 Mercer Street, New York 10012, {\tt varadhan@cims.nyu.edu}. 
}}
\renewcommand{\thefootnote}{}

\vspace{-0.5cm}
\centerline{\textit{Courant Institute New York and WIAS Berlin, Courant Institute New York}}
\vspace{0.2cm}

\begin{center}

\end{center}

\begin{quote}{\small {\bf Abstract: }
In proving large deviation estimates,  the lower bound  for open sets and upper bound for compact sets are essentially  local estimates. On the other hand,   the upper bound for closed sets  is  global and   compactness of  space  or an exponential tightness estimate  is needed to establish it.  In dealing with  the occupation measure $L_t(A)=\frac{1}{t}\int_0^t{\1}_A(W_s) \d s$ of the $d$ dimensional Brownian motion, which is  not positive recurrent, there is no possibility of exponential tightness. The space of probability distributions $\mathcal {M}_1(\R^d)$ can be  compactified by replacing the usual topology of weak convergence by the vague toplogy, where the space  is treated as the dual of continuous functions with compact support. This is essentially the one point compactification of $\R^d$ by adding a point at $\infty$ that results in the compactification of $\mathcal M_1(\R^d)$ by allowing some mass to escape to the point at $\infty$. If one were to  use only test functions that are continuous and vanish at $\infty$ then the compactification results in the space of sub-probability distributions $\mathcal {M}_{\le 1}(\R^d)$ by ignoring the mass at $\infty$.

The main drawback of this compactification is that it ignores the underlying  translation invariance. More explicitly, we may be interested in the space of equivalence classes of orbits $\widetilde{\mathcal M}_1=\widetilde{\mathcal M}_1(\R^d)$    under the  action of the translation group $\R^d$ on  $\mathcal M_1(\R^d)$. There are problems for which it is natural to compactify this space of orbits. We will provide such a compactification, prove a large deviation principle there and give an application to a relevant problem.
}
\end{quote}

\section{Motivation and Introduction} \label{Introduction}

\subsection{Motivation.} We start with the Wiener measure $\P$ on $\Omega= C_0\big([0,\infty); \R^d\big)$ corresponding to the $d$-dimensional Brownian motion $W=(W_t)_{t\geq 0}$ starting from the origin. Our result is motivated by the following set up. Let $L_t$ denote the normalized occupation measure of the Brownian motion until time $t$, i.e.,
\begin{equation}
L_t=\frac 1 t \int_0^t \d s \, \delta_{W_s}.
\end{equation}
This is a random element of $\Mcal_1=\Mcal_1(\R^d)$, the space of probability measures on $\R^d$.
We are interested in the transformed measure
\begin{equation}\label{pathmeasure}
\widehat\P_t(A)=\frac 1{Z_t} \E\bigg\{\1_A\, \exp\{tH(L_t)\}\bigg\}
\end{equation}
with $A$ being a measurable set in the path space of the Brownian motion and
\begin{equation}\label{Hdef}
H(\mu)= \int\int_{\R^d\times \R^d} {V(x-y)} \mu(\d x)\mu(\d y).
\end{equation}
Here $V(\cdot)$ is a continuous function on $\R^d$ vanishing at infinity and 
$$
Z_t=\E\big\{ \exp\{tH(L_t)\}\big\}
$$
is the normalizing constant or the {\it{partition function}}. For $d=3$ and $V(x)= \frac 1{|x|}$, it is known (see \cite{DVP}) that,
\begin{equation}\label{freeenergy}
\lim_{t\to\infty}\frac 1t\log \E \bigg\{\exp\{tH(L_t)\}\bigg\}
=\sup_{\heap{\psi\in H^1(\R^d)}{\|\psi\|_2=1}} 
\Bigg\{\int_{\R^d}\int_{\R^d}\d x\d y\,V(x-y) {\psi^2(x) \psi^2(y)}-\frac 12\big\|\nabla \psi\big\|_2^2\Bigg\},
\end{equation}
where $H^1(\R^d)$ is the usual Sobolev space of square integrable
functions in with their gradient in $L^2(\R^d)$.
For $d=3,\,V(x)=\frac 1{|x|}$, this variational formula has also been analyzed by Lieb (see \cite{L}) who proved that there is a maximizer which is unique except for spatial translations. In other words, if $\mathfrak m$
denotes the set of maximizing densities, then
\begin{equation}\label{shiftunique}
\mathfrak m=  \big\{\mu_0 \star \delta_x\colon x\in \R^3\big\},
\end{equation}
where $\mu_0$ is a probability measure with a density $\psi_0^2$ so that $\psi_0$ maximizes the variational problem \eqref{freeenergy}.
\medskip\noindent

Given \eqref{freeenergy} and \eqref{shiftunique}, we expect that the asymptotic distribution of $L_t$ under $\widehat \P_t$ to be  concentrated around 
 $\mathfrak m$. Indeed, we would like to show
that for very $\eps>0$, 
\begin{equation}\label{tube}
\lim_{t\to\infty}\widehat\P_t\big\{L_t\notin U_\eps(\mathfrak m)\big\} =0
\end{equation}
where $U_\eps(\mathfrak m)$ is a (weak) neighborhood of $\mathfrak m$. In fact, we can write
$$
\begin{aligned}
\widehat\P_t\big\{L_t\notin U_\eps(\mathfrak m)\big\}
&=\frac {\E\left[\1_{L_t\notin U_\eps(\mathfrak m)}\,\exp\big(t H(L_t)\big)\right]} {\E\left[\exp\big(t H(L_t)\big)\right]}
\\
&=\frac  {\E\left[\exp\big(t F(L_t)\big)\right]}{\E\left[\exp\big(t H(L_t)\big)\right]}
\end{aligned}
$$ 
where
\begin{equation}\label{Fdef} 
\begin{aligned}
F(\mu)=
\begin{cases}
H(\mu)={\int\int}_{\R^d\times\R^d} V(x-y){{\mu(\d x)}{\mu(\d y)}} \qquad \mbox{if}\, \mu\notin U_\eps(\mathfrak m)
\\
-\infty \qquad\quad\qquad\qquad\qquad\qquad\qquad\qquad\qquad \mbox{else}.
\end{cases}
\end{aligned}
\end{equation}
Let us pretend that we have a {\it{strong}} Donsker-Varadhan large deviation principle for $L_t$ in $\Mcal_1(\R^d)$ under the weak topology, see \eqref{DVub}-\eqref{DVlb} and the following remarks for a precise definition. Then, using Varadhan's lemma (ignoring the lack of upper semicontinuity of $F$
coming from the singularity for $\small{V(x)=1/{|x|}}$) we could (formally) conclude that $\widehat\P_t\big(L_t\notin U_\eps(\mathfrak m)\big)$ decays exponentially fast in $t$.

However, the lack of a strong large deviation principle turns out a to be crucial issue. To circumvent this problem, the space $\mathcal {M}_1(\R^d)$ has to be ``compactified" . This can be done by replacing the usual topology of weak convergence by the ``vague toplogy", where the space  is treated as the dual of continuous functions with compact support. This is essentially the one point compactification of $\R^d$ by adding a point at $\infty$ that results in the compactification of $\mathcal M_1(\R^d)$ by allowing some mass to escape to the point at $\infty$. If one were to  use only test functions that are continuous and vanish at $\infty$ then the compactification results in the space of sub-probability distributions $\mathcal {M}_{\le 1}(\R^d)$ by ignoring the mass at $\infty$. 

Let us also mention that, for \eqref{freeenergy}, in \cite{DVP}, the lack of compactness of the state space was handled by replacing Brownian motion by Ornstein-Uhlenbeck (O-U) process on $\R^d$ whose occupation measure, unlike Brownian motion, satisfies a strong large deviation principle. Exploiting the positive definiteness of $V(x)=\frac 1{|x|}$ the authors show that the total mass $\E\big\{\exp\{t H(L_t)\}\big\}$ is dominated by the same expectation with respect to the Ornstein-Uhlenbeck process. This monotonicity combined with strong large deviations for the O-U process proves \eqref{freeenergy}. However, no such monotonicity is available to us in the complement of the neighborhood of $\mathfrak m$ (i.e., for the term $\E\left[\1_{L_t\notin U_\eps(\mathfrak m)}\,\exp\big(t H(L_t)\big)\right]$). Another possibility is to replace $\R^d$ by a large torus and ``fold" $L_t$ in the torus and use a similar monotonicity of the total masses (see \cite{DV75}, \cite{BS97}). Although these methods work well for deriving asymptotic behavior of the partition function, questions on the path measures $\widehat\P_t$ can not be handled so well in this manner. In particular, these methods ignore the underlying {\it{translation invariance}} of some relevant models from statistical mechanics, models which depend on shift-invariant functionals of the occupation measures $L_t$, like the functional $H(\mu)= H(\mu\star \delta_x)$ for all $x\in \R^d$, defined in \eqref{Hdef}. Motivated by this, we naturally consider the quotient space 
$$
\widetilde\Mcal_1(\R^d)= \Mcal_1(\R^d)\big/ \sim
$$ 
under spatial shifts and are led to a robust theory of compactification of this space. Let us briefly sketch the main idea here.

\subsection{Translation invariant compactfication: The central idea.}  Note that $\Mcal_1(\R^d)$ fails to be compact in the weak topology for several reasons. For instance, if we take a Gaussian with a very large variance, the mass can spread very thin and totally disintegrate into dust. Also, a mixture like $\frac{1}{2}(\mu \star \delta_{a_n}+\mu\star  \delta_{-a_n})$ splits into two (or more) widely separated pieces as $a_n\to \infty$ . To compactify this space we should be allowed to ``center" each piece separately as well as to allow some mass to be ``thinly spread and disappear". 


The intuitive idea, starting with a sequence of probability distributions $(\mu_n)_n$ in $\R^d$ is to identify a compact region where $\mu_n$ has its largest accumulation of mass. This is given by its {\it{concentration function}} defined by 
$$
q_n(r)= \sup_{x\in \R^d} \mu_n\big(B_r(x)\big).
$$
By choosing subsequences, we can assume that $q_n(r)\to q(r)$ as $n\to\infty$ and $q(r)\to p_1 \in [0,1]$ as $r\to\infty$. Then there is a shift $\lambda_n=\mu_n\star \delta_{a_n}$ which converges along a subsequence vaguely to a sub-probability measure $\alpha_1$ of mass $p_1$. This means $\lambda_n$ can be written as $\alpha_n+\beta_n$ so that $\alpha_n \weak\alpha_1$ weakly and we recover the partial mass $p_1\in [0,1]$. We peel off 
$\alpha_n$ from $\lambda_n$ and repeat the same process for $\beta_n$ to get convergence along a further subsequence. We go on recursively to get
convergence of one component at a time along further subsequences in the space of sub-probability measures, {\it{modulo spatial shifts}}. The picture is, $\mu_n$ roughly concentrates on 
{\it{widely separated}} compact pieces of masses $\{p_j\}_{j\in\N}$ while the rest of the mass $1-\sum_j p_j$ leaks out.

In other words, given any sequence $\widetilde\mu_n$ of equivalence classes in $\widetilde\Mcal_1(\R^d)$, which is the quotient space of $\Mcal_1(\R^d)$ under spatial shifts, there is a subsequence which converges (in a sense which we do not make precise yet) to an element $\{\widetilde\alpha_1,\widetilde\alpha_2,\dots\}$, a collection of equivalence classes of sub-probabilities $\alpha_j$ of masses $0\leq p_j\leq 1$, $j\in\N$.$^{1}$\footnotetext{$^1$For example, let $\mu_n$ be a sequence which is a mixture of three Gaussians, one with mean $0$ and variance $1$, one with mean $n$ and variance $1$ and one with mean $0$ and variance $n$, each with equal weight $\frac 13$. Then the limiting object is the collection $\{\widetilde\alpha_1,\widetilde\alpha_1\}$, where $\widetilde\alpha_1$ is the equivalence class of a Gaussian with variance $1$ and weight $\frac 13$.}
The space of such collections of equivalence classes is the compactfication of $\widetilde\Mcal_1(\R^d)$ and in this space we are able to prove a strong large deviation principle for the distribution of the equivalence classes $\widetilde L_t$ of $L_t$. This, combined with the shift invariant structure of $V(x-y)$, enables us to prove \eqref{tube}.

Finally,  although we were  motivated by asymptotic study of path measures of mean-field type interactions for Brownian motion, it can also be applied to study a wider class of  problems that involve translation invariant functionals of processes with independent increments.

Let us describe the organization of the rest of the article. In section \ref{preliminaries}, we collect some basic facts about weak and vague convergence, introduce a class of relevant test functions and characterize notions 
of {\it{total disintegration of measures}} as well as measures being {\it{widely separated}} in terms of test integrals with respect to the corresponding test functions. In Section \ref{compactification}, we introduce a space
$\widetilde {\mathcal X}$ and a metric $\mathbf D$ giving rise a notion of topology and convergence in this space. Here we also prove that $\widetilde {\mathcal X}$ is the desired compactfication of the quotient space
$\widetilde{\Mcal}_1(\R^d)$. Section \ref{sectionLDP} is devoted to proving a strong large deviation principle for the distribution of the equivalence class $\widetilde L_t$ in $\widetilde {\mathcal X}$ and in Section \ref{application} we provide the application to the asymptotics of the path measures $\widehat \P_t (\widetilde L_t\in \cdot)$.

\section{Topologies on measures}\label{preliminaries}
We denote by $\Mcal_1= {\Mcal_1}(\R^d)$ the space of probability distributions on $\R^d$ and by $\widetilde\Mcal_1= \Mcal_1 \big/ \sim$ the quotient space 
of $\Mcal_1$ under the action of $\R^d$ (as an additive group on $\Mcal_1$). For any $\mu\in \Mcal_1$, its {\it{orbit}} is defined by $\widetilde{\mu}=\{\mu\star\delta_x\colon\, x\in \R^d\}\in \widetilde\Mcal_1$.


\subsection{The weak and the vague topology.} We turn to two natural topologies on $\Mcal_1$. In the {\it{weak topology}}, a sequence $\mu_n$ in $\Mcal_1$ converges to $\mu$,  denoted by  $\mu_n\weak\mu$,  if
\begin{equation}\label{eqweak}
\lim_{n\to\infty}\int_{\R^d} f(x)\mu_n(\d x)=\int_{\R^d} f(x) \mu(\d x),
\end{equation}
for all bounded continuous functions on $\R^d$.  On the other hand, in the {\it{vague topology}} for the convergence of $\mu_n$ to $\mu$, denoted by 
$\mu_n\vague\mu$, we only require (\ref{eqweak}) for continuous functions with compact support. It continues to hold for continuous functions  that tend to $0$ as $|x|\to\infty$. Note that the total mass of
probability measures, which is conserved in the weak convergence, is not necessarily conserved under vague convergence-- a salient feature which distinguishes these two topologies. If we denote by $\Mcal_{\leq 1}=\Mcal_{\leq 1}(\R^d)$ the space of all sub-probability measures (non-negative measures with total mass less than or equal to one), then both topologies carry over to $\Mcal_{\leq 1}$ with the same requirements. 

We collect some standard facts as a lemma which will be relevant for us. 
\begin{lemma}\label{L1.0}
\begin{enumerate}
\item If $\mu_n\vague\mu$ in $\Mcal_{\leq 1}$, then $\mu(\R^d)\le \liminf_{n\to\infty}\mu_n(\R^d)$.
\item  If $\mu_n\vague\mu$ in $\Mcal_{\leq 1}$ and $\mu_n(\R^d)\to \mu(\R^d)$, then $\mu_n\weak\mu$ in $\Mcal_{\le 1}$.
\item While $\Mcal_1$ is a closed subset of the space $\Mcal_{\leq 1}$ in the weak topology, it is dense in $\Mcal_{\leq 1}$ in the vague topology.
\item The space $\Mcal_{\leq 1}$ is compact in the vague topology.
\end{enumerate}
\end{lemma}
We will also need the following elementary lemma.

\begin{lemma}\label{L1.1} If $\mu_n\vague \alpha$ in $\Mcal_{\leq 1}$, then $\mu_n$ can be written as
$\mu_n=\alpha_n+\beta_n$ where $\alpha_n\weak\alpha$ and $\beta_n\vague 0$.  

\end{lemma}

 \begin{proof}
 We will denote by $B(x,r)$ the ball of radius $r>0$ around the point $x\in\R^d$.
 If $\mu_n\vague \alpha$ then
 $$
 \lim_{n\to\infty}\mu_n\big(B(0,r)\big)=\alpha \big(B(0,r)\big),
 $$
 for all but at most countably values of $r$  and $\alpha(\R^d)$ can be recovered as the limit 
  $$
\alpha(\R^d)=\lim_{k\to\infty}\lim_{n\to\infty}\mu_n\big(B(0,r)\big).
  $$
Hence, given any $r>0$, there is $n_r\in \N$ such that for $n\ge n_r$
we have,
$$
\mu_n\big(B(0, r)]\le \alpha(\R^d)+\frac{1}{r}.
$$
Without loss of generality we can assume that $n_r$ is nondecreasing  with $r$. If we define 
$$
R_n=\sup \{r>0: n_r\le n\},
$$ 
then $R_n\to\infty$ and  
$$\mu_n\big(B(0,R_n)\big)\le  \alpha(\R^d)+\frac{1}{R_n}.
$$
 If we take $\alpha_n$  and $\beta_n$ as the restrictions of $\mu_n$ to $B(0,R_n)$ and $B(0,R_n)^c$ respectively, 
$\alpha_n\vague\alpha$ and $\alpha_n(\R^d)\to \alpha(\R^d)$. Therefore, by Lemma \ref{L1.0} part (ii), $\alpha_n\weak\alpha$. Furthermore, for any given $r>0$, eventually  $\beta_n(B(0,r))=0$
and hence $\beta_n\vague 0$.  
  \end{proof}
\medskip
\subsection{The space $\mathcal F$ of test functions.}

For our desired compactification, we need to develop a suitable topology on the quotient space $\widetilde\Mcal_1$ via convergence of test integrals. For this, we first need to characterize a suitable class of continuous functions
(or rather, functionals) on $\widetilde \Mcal_1$.

We fix a positive integer $k\ge 2$. Let ${\mathcal F}_k$ be the space of continuous functions $f: (\R^d)^k \longrightarrow \R$ that are {\it{translation invariant}}, i.e., 
$$
f(x_1+y,\dots,x_k+y)= f(x_1,\dots,x_k) \qquad \forall y,\,\,x_1,\dots, x_k\in \R^d.
$$
and  {\it{vanish at infinity}}, in the sense, 
$$
\lim_{\max_{i\ne j} |x_i-x_j|\to\infty} f(x_1,\dots,x_k) =0,
$$

In other words, $f(x_1,\dots,x_k)$ depends only on the differences $\{x_i-x_j\}_{i,j}$. 
A typical example of a function $f\in \mathcal F_2$ could be $f(x_1,x_2)= V(x_1- x_2)$ where $V(\cdot)$ is a continuous
function such that $V(x)\to 0$ as $|x|\to\infty$. Note that, each $f\in \mathcal F_k$ can interpreted as a continuous function of $k-1$ variables vanishing at infinity.
Hence, for each $k\geq 2$, $\mathcal F_k$  is a separable space under the uniform metric. Hence, if we denote by 
$$
\mathcal F= \cup_{k\geq 2} \mathcal F_k,
$$
then we can choose a countable dense subset for each $\mathcal F_k$ and ordering all of them as a single countable sequence $\{f_r(x_1,\dots,x_{k_r})\colon r\in \N\}$  we obtain a countable  dense subset of  $\mathcal F$. 

For any $\mu\in \Mcal_1$ and $f\in \mathcal F$, we define the  function
$$
\Lambda (f,\mu)=\int_{(\R^d)^k }f(x_1,\ldots, x_k)\mu(\d x_1)\cdots\mu(\d x_k).
$$
Note that, because of translation invariance of $f$, $\Lambda(f, \mu)$ depends only on the orbit $\widetilde\mu\in \widetilde \Mcal_1$ for any fixed $f\in\mathcal F$. 
As it will turn out, these are natural continuous functions to consider on $\widetilde{\mathcal M}_1$. Given any sequence
$(\mu_n)_n$ in $\Mcal_1$, because there is  a countable dense set $\{f_r\}$, by diagonalization one can choose a subsequence such that along the subsequence (denoted again by $\mu_n$), the limit
$$
\lambda (f)=\lim_{n\to\infty}\Lambda (f,\mu_n),\label{eq1.0}
$$
exists. To compactify the space $\widetilde\Mcal_1$ we will determine what the set of  possible limits  are, see Section \ref{compactification}. 

\subsection{Total disintegration of a sequence of measures.}
We say that a sequence $(\mu_n)$ in $\Mcal_{\leq 1}$ {\it{totally disintegrates}} if for any positive $r<\infty$,
\begin{equation*}\
\lim_{n\to\infty}\sup_{x\in\R^d} \mu_n \big(B (x,r)\big)=0.
\end{equation*} 
A typical example of a totally disintegrating sequence $\mu_n$ of measures is a centered Gaussian with covariance matrix $n\, \mathbf{Id}$.

The following facts determine equivalent criteria for total disintegration of a sequence of measures and it is useful to collect them.
\begin{lemma}\label{lemma1.1} 
Let $(\mu_n)_n$ be a sequence in $\Mcal_{\leq 1}$. 
The following facts are equivalent.

\medskip\noindent  
{\bf a}.  There exists a continuous function $V(x)>0$ on $\R^d$, with $\lim_{|x|\to\infty} V(x)=0$, such that
\begin{equation}\label{eq1.1a}
\lim_{n\to\infty}\int \int_{\R^{2d}} V(x-y) \mu_n(\d x)\mu_n(\d y)=0.
\end{equation}

\medskip\noindent
{\bf b}.
\begin{equation}\label{eq1.2}
\lim_{n\to\infty}\sup_{x\in\R^d} \mu_n \big(B (x,r)\big)=0.
\end{equation}

\medskip\noindent
{\bf c}.  For any continuous function $V(x)$ with $\lim_{|x|\to\infty} V(x)=0$,
\begin{equation}\label{eq1.3}
\lim_{n\to\infty} \sup_{x\in \R^d}\int V(x-y) \mu_n(\d y)=0.
\end{equation}
\end{lemma}

\medskip\noindent
{\bf d}. For any continuous function $V(x)$ with $\lim_{|x|\to\infty} V(x)=0$,
\begin{equation}\label{eq1.4}
\lim_{n\to\infty}\int\int_{\R^{2d}} V(x-y) \mu_n(\d x)\mu_n(\d y)=0.
\end{equation}

Let the sequence  $\mu_n$  of measures satisfy any of the above.  Then for any $k\ge 2$ and $f\in{\mathcal F}_k$
\begin{equation}\label{eq1.4a} 
\lim_{n\to\infty}\int\dots\int_{\R^{dk}} f(x_1,\ldots,x_k) \mu_n(\d x_1)\cdots \mu_n(\d x_k)=0.
\end{equation}
\begin{proof}

\smallskip\noindent
$a)\implies b)$. Let $r>0$ be given. Since $V(x)>0$  and continuous, there exists $\delta>0$ such that $V(x)\ge \delta$ on $B(0,2r)$. Then,
$$   
\int \int_{\R^{2d}} V(x-y)\mu_n(\d x)\mu_n(\d y) \ge  \delta \int_{|x-y|\le 2r}\mu_n(\d x)\mu_n(\d y)\ge \delta \sup_{x\in \R^d} \big\{\mu_n\big(B(x,r)\big)\big\}^2.
$$

\smallskip\noindent 
$b)\implies c)$.
Let   $\eps_M=\sup_{|x|\ge M}|V(x)|$. Then $\lim_{M\to\infty}\eps_M=0$ and
\begin{align*}
\sup_{x\in \R^d}\int |V(x-y)|\mu_n(\d y)&\le \sup_{x\in \R^d} \int_{B(x,M)} |V(x-y)|\mu_n(\d y)+\sup_{x\in \R^d}\int_{B(x,M)^c} |V(x-y)|\mu_n(\d y)\\
&\le \|V\|_\infty \sup_{x\in \R^d} \mu_n [B(x,M)]+\eps_M.
\end{align*}
Therefore
$$
\limsup_{n\to\infty}\sup_{x\in \R^d}\int |V(x-y)|\mu_n(\d y)\le \eps_M,
$$
for any $M$. Since $\eps_M\to 0$ if we let $M\to\infty$, we get the claim.

\smallskip\noindent
$c)\implies d)$. Observe that, since $\mu_n(\R^d)\le 1$,
$$ 
\int\int_{\R^{2d}} V(x-y)\mu_n(\d x)\mu_n(\d y) \le \sup_{x\in \R^d}\int V(x-y) \mu_n(\d y).
$$

\smallskip\noindent
$d)\implies a)$. This is obvious.

\medskip\noindent
For the last part, for  $k>2$  we define 
$$ W(x_1, x_2)=\sup_{x_3,\ldots,x_k} |f(x_1,\ldots, x_k)|.
$$
Note that $W\in{\mathcal F}_2$ and so it is of the form $V(x_1-x_2)$. Since
$$
\begin{aligned}
\int\dots\int_{\R^{dk}} |f(x_1,\ldots,x_k)|\mu_n(\d x_1)\cdots \mu_n(\d x_k)&\le \int\int_{\R^{2d}} W(x_1,x_2) \mu_n(\d x_1)\mu_n(\d x_2)\\
&=\int\int_{\R^{2d}} V(x-y)\mu_n(\d x)\mu_n(\d y),
\end{aligned}
$$
the lemma is proved.
\end{proof}

\subsection{Widely separated sequences of measures.} 

We now need a working definition of two {\it{widely separated}} sequence of measures.
We say that two sequences $(\alpha_n)_n$ and $(\beta_n)_n$ in $\Mcal_{\leq 1}$ 
are widely separated, if for some strictly positive function $V$ on $\R^d$ which is continuous and vanishes at infinity,
\begin{equation}\label{eq1.6}
\lim_{n\to\infty}\int V(x-y) \alpha_n(\d x)\beta_n(\d y)=0.
\end{equation}
Note that if a sequence $(\mu_n)_n$ in $\Mcal_{\leq 1}$ satisfies \eqref{eq1.1a}, then, because of \eqref{eq1.3}, it is widely separated from any arbitrary sequence of measures in $\Mcal_{\leq 1}$.
\begin{lemma}\label{lemma2.3}
Let $(\alpha_n)_n$ and $(\beta_n)_n$ be two widely separated sequences in $\Mcal_{\leq 1}$.
Then,
 \begin{enumerate}
 \item For any continuous function $W$ in $\R^d$ vanishing at infinity
$$
\lim_{n\to\infty}\int W(x-y) \alpha_n(\d x)\beta_n(\d y)=0.
$$
\item For every $k\ge 2$ and $f\in {\mathcal F}_k$,
\begin{equation}\label{widesep}
\begin{aligned}
\lim_{n\to\infty}\bigg| \int f(x_1,\ldots,x_k)&\prod_{i=1}^k [\alpha_n+\beta_n](\d x_i)-\int f(x_1,\ldots,x_k)\prod_{i=1}^k \alpha_n(\d x_i)\\
&-\int f(x_1,\ldots,x_k)\prod_{i=1}^k \beta_n(\d x_i)\bigg|=0.
\end{aligned}
\end{equation}
\end{enumerate}
\end{lemma}
\begin{proof}
Let $W$ be any continuous function $\R^d$ vanishing at infinity. 
Since $(\alpha_n)_n$ and $(\beta_n)_n$ are widely separated, for some strictly positive $V$ which is continuous and vanishes at infinity,
$$
\lim_{n\to\infty}\int V(x-y) \alpha_n(\d x)\beta_n(\d y)=0.
$$
Furthermore, given any $\eps>0$, there is a constant $C_\eps>0$ such that
$$
|W(x)|\le C_\eps V(x)+\eps.
$$
Then
$$
\limsup_{n\to\infty}\int |W(x-y)|\alpha_n(\d x)\beta_n(\d y)\le C_\eps \limsup_{n\to\infty}\int V (x-y)\alpha_n(\d x)\beta_n(\d y)+\eps=\eps.
$$
This proves the first part (i). 

For the second part (ii), if we take $k=2$ and expand the product 
$$
\prod_{i=1}^2 (\alpha_n+\beta_n) (\d x_i),
$$
it is seen that all the cross terms are controlled by \eqref{eq1.6} and are negligible, by the first part (i), as $f(x_1, x_2)= W(x_1-x_2)$ for some continuous $W$ vanishing at infinity. 
The general case $k\geq 3$ follows easily. 
\end{proof}

\section{Compactification of $\widetilde \Mcal_1$: The space $\widetilde{\mathcal X}$} \label{compactification}

We turn to the central issue of $\Mcal_1$ failing to be compact in the weak topology. As mentioned before some typical reasons for this could be as follows: The location of the mass  can shift away to $\infty$ as in $\mu_n=\mu\ast\delta_{a_n}$ with $a_n\to\infty$, or it can split into two (or more)  pieces like in $\mu_n=\frac{1}{2}[\mu\ast\delta_{a_n}+\mu\ast\delta_{-a_n}]$, or it can also totally disintegrate into dust like a Gaussian with  a large variance. One imagines, in the limit, an empty, finite or countable collection $I$ 
of mass distributions $\{\alpha_i\colon \, i\in I\}$ that are widely separated with total mass $\sum_{i\in I} \alpha_i (\R^d)=p\le1$ and the remaining mass $1-p$ having totally disintegrated. 
Therefore, a natural ``compactification" could be a space $\widetilde{\mathcal X}$ of empty, finite or countable collections of orbits $\{{\widetilde \alpha}_i\colon \, i\in I\}$ of sub-probability distributions $\alpha_i$ having masses $p_i$ with $p=\sum_i p_i\le 1$. 

\subsection{The space $\widetilde{\mathcal X}$ and a metric $\mathbf D$.} 

We define
\begin{equation}\label{Xdef}
\widetilde{\mathcal X}= \bigg\{\xi\colon\,\xi=\{\widetilde{\alpha_i}\}_{i\in I}, \, \alpha_i\in \Mcal_{\leq 1}, \, \sum_{i\in I} \alpha_i(\R^d) \leq 1\bigg\}.
\end{equation}

We make some remarks about the above definition.
\begin{remark}\label{rmk1}
 First note that, in order to keep notation short, we suppressed the fact that the index set $I$ above ranges over empty, finite or countably many collections. 
 Furthermore, we will write any typical element $\xi\in \widetilde{\mathcal X}$ as $\xi=\{\alpha_i\}$ with the understanding that either the collection is empty or $i$ ranges over a finite
 or countable set.
\end{remark}
\begin{remark}
 Note that any element $\alpha \in \Mcal_{\leq 1}$ in the orbit $\widetilde\alpha$ has the same total mass $\alpha(\R^d)$. 
Hence, for any element $\xi=\{\widetilde\alpha_i\}\in\widetilde{\mathcal X}$, $p_i=\alpha_i(\R^d)$ will denote the total mass of any candidate $\alpha_i$ in the orbit $\widetilde\alpha_i$ and 
$\sum_{i} p_i=p \leq 1$.  If the collection is empty, then $p=0$ vacuously.
\end{remark}
\begin{remark}\label{rmk2}
Note that in any element $\xi=\{\widetilde\alpha_i\}$ of $\widetilde{\mathcal X}$, an orbit $\alpha_i$ could be repeated more than once. We call the number of occurrences of an orbit in an element $\xi$ its {\it{multiplicity}}.
\end{remark}

We now introduce a metric on $\widetilde{\mathcal X}$  that corresponds to the  convergence (\ref{eq1.0}).  Recall the class of functions 
$\mathcal F_k$ for $k\geq 2$ and $\mathcal F= \cup_{k\geq 2}\mathcal F_k$. 
We want a sequence $(\xi_n)_n$ to converge to $\xi$ in the space $\widetilde{\mathcal X}$ under the desired metric, if the sequence
$$
\Lambda(f,\xi_n)=\sum_{\widetilde{\alpha}_n\in\xi_n}\int f(x_1,\ldots,x_k)\alpha_n(\d x_1)\cdots \alpha_n(\d x_k)
$$
converges to the corresponding expression 
$$
\Lambda (f,\xi)=
\sum_{\widetilde\alpha\in\xi}\int f(x_1,\ldots,x_k)\alpha(\d x_1)\cdots \alpha(\d x_k)
$$
for every $f\in\mathcal F$. Recall that the value of $\int f(x_1,\ldots,x_k)\alpha(\d x_1)\cdots \alpha(\d x_k)$ 
depends only on the orbit $\widetilde\alpha$ since $f$ is translation invariant.
We also remark that if $\xi$ is empty then $\Lambda(f,\xi)=0$ for all $f \in \mathcal F$. 

For any $\xi_1,\xi_2\in {\widetilde{\mathcal X}}$, we define
\begin{equation}\label{Ddef}
\mathbf D(\xi_1,\xi_2)=\sum_{r=1}^\infty\frac{1}{2^r}\frac{1}{1+\|f_r\|_\infty} \bigg|\sum_{\widetilde\alpha\in\xi_1}\int f_r(x_1,\ldots,x_{k_r})\prod_{i=1}^{k_r}\alpha(\d x_i)-\sum_{\widetilde\alpha\in\xi_2}\int f_r(x_1,\ldots,x_{k_r})\prod_{i=1}^{k_r}\alpha(\d x_i)\bigg|,
\end{equation}
for a countable sequence  of functions $\{f_r(x_1,x_2,\ldots, x_{k_r})\}_{r\in \N}$ which is dense in $\mathcal F$. Here is our first main result.

\begin{theorem}
$\mathbf D$ is a metric on $\widetilde{\mathcal X}$. 
\end{theorem}
\begin{proof}
Note that to prove $\mathbf D$ is a metric the only nontrivial part that we need to show is that, two collections $\xi_1$ and $\xi_2$ are
identical if $\mathbf D(\xi_1, \xi_2)=0$. For this it is enough to show if for every $k\geq 2$ and  every $f$ in ${\mathcal F}_k$,
\begin{equation}\label{claim0}
\sum_{\widetilde\alpha\in\xi_1}\int f(x_1,\ldots,x_{k})\prod_{i=1}^{k}\alpha(\d x_i)=\sum_{\widetilde\alpha\in\xi_2}\int f(x_1,\ldots,x_{k})\prod_{i=1}^{k}\alpha(\d x_i).
\end{equation}
then $\xi_1=\xi_2$. We prove this into three steps.

{\bf{Step 1:}} First we  show that, if \eqref{claim0} holds, then for every integer $r\ge 1$,
\begin{equation}\label{claim1}
\sum_{\widetilde\alpha\in\xi_1} \bigg\{\int f(x_1,\ldots,x_k)\prod_{i=1}^k \alpha(\d x_i)\bigg\}^r=\sum_{\widetilde\alpha\in\xi_2} \bigg\{\int f(x_1,\ldots,x_k)\prod_{i=1}^k \alpha(\d x_i)\bigg\}^r.
\end{equation}
This is certainly true for $r=1$. For $r=2$, we take a sequence $g_N$ of functions of $2k$ variables defined by
$$
g_N(x_1,x_2,\ldots, x_{2k})=f(x_1,x_2,\ldots, x_k)f(x_{k+1},x_{k+2},\ldots, x_{2k})\varphi(N^{-1}(x_1-x_{k+1})),
$$ 
 where  $0\le \varphi\le 1$ is equal to  $1$ inside  a ball of radius $1$ and is truncated smoothly to be $0$ outside a ball of radius $2$.  Letting  $N\to \infty$,  for any $\alpha\in \Mcal_{\leq 1}$, by the bounded convergence theorem,
$$
\int g_N(x_1,\ldots x_{2k}) \, \prod_{i=1}^{2k} \alpha(\d x_i)\to \bigg\{\int f(x_1,x_2\ldots,x_k) \,\prod_{i=1}^k \alpha(\d x_i)\bigg\}^2,
$$
and we obtain
\begin{equation*}
\sum_{\widetilde\alpha\in\xi_1} \bigg\{\int f(x_1,\ldots,x_k)\prod_{i=1}^k\alpha(\d x_i)\bigg\}^2=\sum_{\widetilde\alpha\in\xi_2} \bigg\{\int f(x_1,\ldots,x_k)\prod_{i=1}^k \alpha(x_i)\bigg\}^2.
\end{equation*}
The general case for any $r\in \N$ follows from a similar argument. 

{\bf{Step 2:}} We note that if \eqref{claim1} holds for every $r\in \N$, we can identify for each $\alpha\in \Mcal_{\leq 1}$ the values of 
\begin{equation}\label{list}
 \int f(x_1,\ldots,x_k)\prod_{i=1}^k\alpha(\d x_i),
\end{equation}
for $f\in \mathcal F_k$ and $k\ge 2$. It follows that if \eqref{claim0} holds for any two elements $\xi_1$ and $\xi_2$, then for every $f\in{\mathcal F}$, the list of values 
\eqref{list} for $\widetilde\alpha\in\xi_1$ is the same as the list from $\xi_2$. 


However, this is not enough. We need to show that if \eqref{claim0} holds for any two elements $\xi_1$ and $\xi_2$, then every $\widetilde{\alpha}\in \xi_1$ occurs in $\xi_2$ with the same multiplicity (see Remark \ref{rmk2} for the definition of multiplicity of an orbit).  

Let us denote by $S(f,\xi)$ the set of values of $\Lambda (f,{\widetilde \mu})$ as ${\widetilde \mu}$ varies over $\xi$.
We have matched for $\xi=\xi_1$ and $\xi_2$ the set of values $S(f,\xi_1)$ and $S(f,\xi_2)$. The next step is to show
that if  we pick an orbit ${\widetilde \mu}_1$ in $\xi_1$, the set of values of $\Lambda (f,{\widetilde \mu}_1)$ can actually be matched with the set of values  $\Lambda (f,{\widetilde \mu}_2)$ of some ${\widetilde \mu}_2\in\xi_2$ i.e. a single  choice ${\widetilde \mu}_2\in \xi_2$ can be made to work for all $f\in \cup_k {\mathcal F}_k$. In other words, if we define for
each ${\widetilde \mu}_1\in \xi_1$ and ${\widetilde \mu}_2\in \xi_2$
$$
    C_k\big({\widetilde\mu_1},{\widetilde\mu_2}\big)=\bigg\{f\in{\mathcal F}_k\colon\, \Lambda (f,{\widetilde \mu}_1)=\Lambda (f,{\widetilde \mu}_2)\bigg\},
$$
then we have for each ${\widetilde\mu_1}$ 
$$
{\mathcal F}_k=\cup_{{\widetilde\mu_2}\in\xi_2}\,\,C_k\big({\widetilde\mu_1},{\widetilde\mu_2}\big).
$$
Each $C_k\subset {\mathcal F}_k$ is a closed subset of a complete metric space and we have a countable union. By the Baire  category theorem  at least  one $C_k$ has an interior.  But if two linear functionals agree on an open set they agree everywhere. Therefore there exists $\mu_2\in \xi_2$ such that
\begin{equation}\label{xxx}
{\mathcal F}_k=C_k\big({\widetilde\mu_1},{\widetilde\mu_2}\big).
\end{equation}
The choice of ${\widetilde \mu}_2$ may still depend on $k$. We need to show that \eqref{xxx} holds for some ${\widetilde\mu_2}$ 
 for all $k$.  We note that any function $f(x_1,x_2,\ldots, x_{k-1})\in {\mathcal F}_{k-1}$ is a limit of 
 $$
 g(x_1,x_2,\ldots, x_{k-1},x_k)=f(x_1,x_2,\ldots, x_{k-1})\varphi(x_{k-1}-x_k)\in  {\mathcal F}_k
 $$ 
 as the continuous function $\varphi$ with compact support tends boundedly to $1$. Therefore if $\Lambda (f,{\widetilde \mu_1})=\Lambda (f,{\widetilde \mu_2})$ on  ${\mathcal F}_k$, they agree on ${\mathcal F}_{k-1}$ as well. In particular if $\Lambda (f,{\widetilde \mu_1})=\Lambda (f,{\widetilde \mu_2})$ on ${\mathcal F}_k$ for infinitely many values of $k$, then they agree for all values of $k$. We note that by allowing $\varphi$ to tend to $1$,
 if $\Lambda (f,{\widetilde \mu_1})=\Lambda (f,{\widetilde \mu_2})$ on ${\mathcal F}_2$, then ${ \mu}_1 $ and ${\mu}_2$ have the same mass. Assuming the mass to be positive, there can only be a finite number of possibilities for $\mu_2$ since the total sum is at most $1$. There is then a $\mu_2$ that works for an infinite number of values of $k$ and consequently for all $k$. We can then peel off matching pairs  and proceed with what is left.  If we are careful to remove at each stage measures with the largest masses from $\xi_1$ and $\xi_2$, we will exhaust both $\xi_1$ and $\xi_2$ (it may take a countable number of steps). 

\medskip

{\bf{Step 3:}} Now we have to recover the orbit of $\mu\in \mathcal M_{\le 1}(\R^d)$ from the value 
$$\int f(x_1,\ldots,x_k)\prod_{i=1}^k\mu(\d x_i)
$$ 
for $f\in \mathcal F_k$. We can let $f$ converge boundedly  to $\exp\{\sum_{i=1}^k \sqrt{-1}\langle t_i ,x_i\rangle\}$ provided $\sum_i t_i=0$. In other words we can determine for the characteristic functions $\{\phi(t)e^{\sqrt{-1}\langle t,a\rangle}\}$ of $\widetilde{\alpha}\in {\mathcal X}$, the value of   $\prod_{i=1}^k \phi(t_i)$ for all $\{t_i\}$ with $\sum_i t_i=0$. 

The following calculation  will complete the proof. Let $\phi(\cdot)$ and $\psi(\cdot)$ be two characteristic functions such that  $\prod_{i=1}^k \phi(t_i)= \prod_{i=1}^k \psi(t_i)$ for all $\{t_i\}$ with $\sum_i t_i=0$. In particular $|\phi(t)|^2=\phi(t)\phi(-t)=\psi(t)\psi(-t)=|\psi(t)|^2$. 
Let $G=\{t: |\phi(t)|=|\psi(t)|\not=0\}$. Write $\phi(t)=\psi(t)\chi(t)$ on $G$. $G$ is  a symmetric open set containing   $0$. For any $k$ and  $t_1,\ldots,t_k\in G$ such that $\sum_{i=1}^k t_i =\tau\in G$, we have

$$\prod_{i=1}^k\chi(t_i)\chi(-\tau)=1$$
Noting that $\chi(\tau)=\overline {\chi(-\tau )}$, we find that $\chi(t_1+t_2+\cdots+t_k)=\prod_{i=1}^k \chi(t_i)$ provided, $\{t_i\}$ as well as $t_1+\cdots+t_k$ are all in $G$ which contains a neighborhood of $0$. It is now standard to show that  for some $a\in \R^d$, $\chi(t)=e^{\sqrt{-1}\,\langle a,t\rangle}$ near $0$ and since $\chi(kt)=(\chi(t))^k$  the proof is complete.
\end{proof}

\subsection{Completion under the metric $\mathbf D$ and the compactification.}

Henceforth, the metric $\mathbf D$ will define the topology on the space $\widetilde {\mathcal X}$. Recall that the space of orbits $\widetilde{\mathcal M}_1$ is canonically embedded in $\widetilde{\mathcal X}$.

\begin{theorem}\label{T2.2}
The set of orbits $\widetilde{\mathcal M}_1(\R^d)$  is dense in $\widetilde{\mathcal X}$. Furthermore, given any sequence $(\widetilde\mu_n)_n$ in $\widetilde{\mathcal M}_1(\R^d)$, there is a subsequence that converges to a limit in $\widetilde{\mathcal X}$. Hence $\widetilde{\mathcal X}$ is a compactification of $\widetilde{\mathcal M}_1(\R^d)$. It is then also the completion under the metric $\mathbf D$ of the totally bounded space $\widetilde{\mathcal M}_1(\R^d)$.
\end{theorem}

\begin{proof}
We prove the theorem in two main steps.

\medskip\noindent
{\bf{Step 1:}} First we show that $\widetilde{\mathcal M_1}$ is dense in $\widetilde{\mathcal X}$. Given any $\xi=\big\{\widetilde\alpha_i\colon\, i\in I\}\in \widetilde {\mathcal X}$, we would like to have a sequence 
$(\widetilde\mu_n)_n$ in $\widetilde \Mcal_1$ which converges to $\xi\in \widetilde X$. This can be done if we take ``distant shifts" of $\mu_n$ weighted by corresponding masses $p_i$ of $\alpha_i$. Any remaining mass
$1-\sum_i p_i$ can be filled by a Gaussian with a large variance (leading to ``total disintegration" of mass $1-\sum_i p_i$). The convex combination of all these measures will approximate $\xi$ in $\widetilde {\mathcal X}$.

Indeed, let $\xi=\{\widetilde\alpha_i\colon i\in I\}\in \widetilde{\mathcal X}$ be given. If it is an infinite collection, then for every $\eps>0$ we can pick a finite sub-collection $\{\alpha_1,\dots,\alpha_n\}$ such that the remaining total masses $\sum_{j>n} \alpha_j(\R^d)$ add up to at most $\eps>0$. Since for any $\alpha\in \Mcal_{\leq 1}$ and $f\in \mathcal F_k$, 
$$
\int f(x_1,\ldots,x_k)\prod_{i=1}^k\alpha(\d x_i)\le \|f\|_\infty \, \big(\alpha(\R^d)\big)^k\le \|f\|_\infty \, \alpha(\R^d),
$$
we infer that 
\begin{equation}\label{largen}
\sum_{j>n} \int f(x_1,\dots,x_k) \prod_{i=1}^k \alpha_j (\d x_i) \leq \|f\|_\infty \sum_{j>n} \alpha_j(\R^d) \leq \eps \|f \|_\infty.
\end{equation}
Let us denote by $p_j= \alpha_j(\R^d)$ for $j=1,\dots, n$ and choose spatial points $a_1,\dots, a_n\in \R^d$ so that $\inf_{i\ne j} |a_i- a_j|\to \infty$. Also,
for any $M>0$, let $\lambda_M$ be a Gaussian in $\R^d$ with mean $\mathbf 0 \in \R^d$ and 
covariance matrix $M\,\mathrm{\bf{Id}}\in \R^{d\times d}$. Since the family of measures $\{\lambda_M\}_{M>0}$ totally disintegrates, by Lemma \ref{lemma1.1} and \eqref{eq1.4a}, for any $k\ge 2$ and $f\in {\mathcal F}_k$,
\begin{equation}\label{Gaussdisint}
\lim_{M\to\infty}\int f(x_1,\ldots,x_k) \prod_{i=1}^k \lambda_M(\d x_i)=0.
\end{equation}
Then for the convex combination 
\begin{equation}\label{convexcomb}
\mu_n^{\ssup{a_1,\dots,a_n, M}}= \mu_n:=  \sum_{j=1}^n \alpha_j\star  \delta_{a_j} + \bigg(1-\sum_{j=1}^n p_j\bigg) \lambda_M,
\end{equation}
we conclude that, for any $k\geq 2$ and $f\in \mathcal F_k$,
$$
 \int f(x_1,\ldots,x_k)\prod_{i=1}^k\mu_n (\d x_i)\longrightarrow \sum_{j=1}^n \int f(x_1,\ldots,x_k)\prod_{i=1}^k\alpha_j (\d x_i)
$$
as $\inf_{i\ne j} |a_i- a_j|\to \infty$ and $M\uparrow \infty$, by \eqref{widesep} and \eqref{Gaussdisint} (masses that are far away from each other do not interact and masses that are too thinly spread do not count). Therefore,
by \eqref{largen} and the definition of the metric $\mathbf D$ (recall \eqref{Ddef}), the sequence of orbits $(\widetilde\mu_n)_n$ converges to $\xi$ in $\widetilde{\mathcal X}$.

\medskip\noindent
{\bf{Step 2:}} We show that any sequence $(\widetilde{\mu}_n)_n$ in $\widetilde{\mathcal X}$
has a subsequence that converges to some $\xi\in {\widetilde{\mathcal X}}$. We need to collect some facts.

Let $\mu \in \Mcal_{\leq 1}(\R^d)$. The {\it{concentration function}} of $\mu$ is defined as,
\begin{equation}\label{eq2.4}
q_\mu(r)=\sup_{x\in \R^d} \mu\big(B(x,r)\big),
\end{equation}
for any $r>0$. Then $\lim_{r\to\infty}q_\mu(r)=\mu(\R^d)$.

If $(\mu_n)_n$ is a sequence in $\Mcal_{\leq 1}(\R^d)$, and $q_n(r)$ is the concentration function of $\mu_n$, we can, by choosing a subsequence (which we suppress in the notation) if needed, assume that for any $r>0$, 
$$
\lim_{n\to\infty} q_n(r)=q(r),
$$
exists and also
$$
\lim_{n\to\infty}\mu_n(\R^d)=p \in [0,1].
$$
If $q=\lim_{r\uparrow\infty}q(r)$, then always $q\le p$. 

If $q=0$, we have for every $r>0$
$$
\lim_{n\to \infty} \sup_{x\in \R^d} \mu_n \big(B(x,r)\big)=\lim_{n\to\infty} q_n(r)= 0
$$
and hence by Lemma \ref{lemma1.1}, \eqref{eq1.4a} and the definition of the metric $\mathbf D$ (recall \eqref{Ddef}), ${\widetilde\mu}_n\to 0$ in  $\widetilde\skrix$. 

If $q>0$, then taking a suitable translation $a_n\in \R^d$, we can assume that $\lambda_n=\mu_n\star\delta_{a_n}$ satisfies, for some $r>0$,
\begin{equation}\label{lambda}
\lambda_n \big(B(0,r)\big)\ge {q}/{2},
\end{equation}
for all sufficiently large $n$.  Let us assume, by choosing a subsequence if needed, $\lambda_n\vague \alpha$. Then $\alpha(\R^d)\ge \frac{q}{2}$. According to Lemma \ref{L1.1}, we can express
$\lambda_n=\alpha_n+\beta_n$ where $\beta_n\vague 0$ and $\alpha_n\weak\alpha$. Lemma \ref{lemma2.3} implies that for $V\in {\mathcal F}_2$
$$
\lim_{n\to\infty}\int V(x-y) \alpha_n(\d x)\beta_n(\d y)=0.
$$
This property is valid after translating back by $\delta_{-a_n}$ and $\mu_n$
has the same decomposition in terms of the shifted $\alpha_n\ast\delta_{-a_n}$  and $\beta_n\ast\delta_{-a_n}$.  We will denote them again by $\alpha_n$ and $\beta_n$.  
We remark that if $q=p$, then $\lambda_n=\mu_n\star \delta_{a_n}$ converges weakly to $\alpha$ and $\beta_n$ can be taken to be $0$. To see this, choose $r>0$ so that \eqref{lambda} holds. Furthermore, note that given any $\eps>0,$ there are translations $b_{n,\eps}$ such that, for some $r_\eps$,  $(\mu_n\star\delta_{b_n}) [B(0, r_\eps)]\ge p-\eps$ for large enough $n$.  The sets $B(-a_n,r)$ and $B(-b_n, r_\eps)$ can not be disjoint,  because if they were, there combined total mass would exceed $p$ (recall \eqref{lambda}). Therefore
$|a_n-b_n|\le r+r_\eps$. This implies $B(-a_n,r+2r_\eps)\supset B(-b_n,r_\eps)$ and $\lambda_n[B(0, r+2r_\eps)]\ge p-\eps$. This 
shows that $\lambda_n$ is a tight family of measures and (choosing a subsequence if needed) $\lambda_n\weak\alpha$ for some $\alpha\in \Mcal_{\leq 1}(\R^d)$ and $\beta_n$ can be taken as $0$. Hence, again by 
definition of the metric $\mathbf D$ (recall \eqref{Ddef}), ${\widetilde\mu}_n\to \widetilde\alpha$ in $\widetilde\skrix$.

Let us now start with a sequence $(\mu_n)_n$ in $\Mcal_1(\R^d)$. We want to prove that the sequence $(\widetilde{\mu}_n)_n$ in $\widetilde{\mathcal X}$
has a subsequence that converges to some $\xi\in {\widetilde{\mathcal X}}$. Hence, to begin with $p=1$ and $0\le q\le 1$. By the remarks made above, if $q=0$, then $\widetilde{\mu}_n\to 0$ in $\widetilde{\mathcal X}$. If $q=1$, then $\alpha_n=\mu_n$ and ${\widetilde \mu}_n\to {\widetilde \alpha}$ in ${\widetilde{\mathcal X}}$.

\medskip
If $0<q<1$, we can, for some sequence $(a_n)_n\subset\R^d$, represent $\mu_n=\alpha_n+\beta_n$. $\alpha_n$ so that
\begin{itemize}
\item 
$$
\alpha_n\ast\delta_{a_n}\Rightarrow \alpha.
$$ 
\item For every $V\in \mathcal F_2$,
$$
\lim_{n\to \infty} \int V(x_1-x_2)\alpha_n(\d x_1)\beta_n(\d x_2)=0.
$$
\item For every $r>0$,
$$
\lim_{n\to\infty}q_{\beta_n}(r)\le \min\big\{1-\frac{q}{2}, q\big\}.
$$
\end{itemize}
 The last inequality requires a remark. Since $\beta_n\le \mu_n$ we have 
 $q_{\beta_n}(r)\le q_{\mu_n}(r)$ for every $r$. Mass of $\frac{q}{2}$ has been removed in the limit from $\mu_n$ by \eqref{lambda}. What is left can  in the limit have mass at most $1-\frac{q}{2}$. 
 
 \medskip
 We repeat the procedure with $\beta_n$.  Either the process goes on forever or terminates at some finite stage. If it terminates at a finite stage we would have the decomposition
 \begin{equation}\label{mudecompose}
 \mu_n=\sum_{j=1}^k \alpha_{n}^{\ssup j}+\gamma_n \qquad k\in \N,
 \end{equation}
 that will satisfy 
 \begin{itemize}
 \item For $j=1,\dots,k$,
 $$
 \lim_{n\to\infty}\alpha_{n}^{\ssup j}\ast a_{n}^{\ssup j}\Rightarrow \alpha_j.
 $$
 \item For $ i\not=j$ and $V\in {\mathcal  F}_2$, 
 $$
\lim_{n\to\infty} \int V(x_1-x_2)\alpha_{n}^{\ssup i}(\d x_1)\alpha_{n}^{\ssup j}(\d x_2)=0.
 $$
 \item For every $r>0$, $q_{\gamma_n}(r)\to 0$  and
 $$
\lim_{n\to\infty} \int V(x_1-x_2)\gamma_n(\d x_1)\gamma_n(\d x_2)=0.
 $$
 \end{itemize}
Clearly ${\widetilde \mu}_n$ converges to $\xi=\{{\widetilde\alpha}_1,\ldots, {\widetilde\alpha}_k\}$ in ${\widetilde{\mathcal X}}$.

If the process continues forever, we have for each $k\in \N$ a decomposition as above. Inductively, starting from
$\beta_{n,0}=\mu_n$,  we define according to Lemma \ref{L1.1} $\beta_{n,j}=\alpha_{n,j+1}+\beta_{n,j+1}$ so that $\alpha_{n,j}\Rightarrow \alpha_j$. Let  $p_j=\lim_{n\to\infty}\beta_j(\R^d)$ and $q_j=\lim_{r\to\infty}\lim_{n\to\infty}q_{\beta_{n,j}}(r)$. Since $\alpha_j(\R^d)\ge \frac{q_j}{2}$, and $\sum_j\alpha_j(\R^d)\le 1$, it follows that $q_j\to 0$ as $j\to\infty$. Fix any $F\in {\mathcal F}_k$. Then,
proceeding inductively in $j$,  
 
 \begin{align*}
\int F(x_1,\ldots, x_k)\mu_n(\d x_1)\cdots\mu_n(\d x_k)&=\sum_{i=1}^j\int F(x_1,\ldots, x_k)\alpha_{n,i}(\d x_1)\cdots\alpha_{n,i}(\d x_k)\\
&\qquad+\int F(x_1,\ldots, x_k)\beta_{n,j}(\d x_1)\cdots\beta_{n,j}(\d x_k)
\end{align*}
Since $q_j\to 0$ and the orbits ${\widetilde \alpha_{n,j}}$ converge to
${\widetilde \alpha_{j}}$ in ${\widetilde {\mathcal X}}$, the theorem is proved.
\end{proof}
We end this section with an immediate corollary which will be of use later.

\begin{cor}\label{continuity}
Let $(\widetilde\mu_n)_n$ be a sequence in $\widetilde{\mathcal X}$ so that $\widetilde\mu_n\to \xi= \{\widetilde\alpha_j\} \in \widetilde{\mathcal X}$. Then, for any $V\in \mathcal F_2$,
$$
\lim_{n\to\infty} \int\int_{\R^d\times \R^d} V(x-y) \mu_n(\d x)\mu_n(\d y)= \sum_j \int\int_{\R^d\times \R^d} V(x-y) \alpha_j(\d x)\alpha_j(\d y).
$$
In other words, the functional 
$$
H(\widetilde\mu)=\int\int_{\R^d\times \R^d} V(x-y) \mu(\d x)\mu(\d y) \qquad\qquad\mu\in\Mcal_1(\R^d),
$$
is continuous on $ \widetilde{\mathcal X}$.
\end{cor}

\section {Large Deviation Principles in the compact space $\widetilde{\mathcal X}$} \label{sectionLDP}
Recall that we started with Wiener measure $\P$ on $\Omega=C[[0,\infty);\R^d]$ corresponding to the $d$-dimensional Brownian motion $W$ starting from the origin with
 $$
L_t(A)=\frac{1}{t}\int_0^t {\1}_A(W(s))\d s \qquad A\subset \R^d
$$
denoting its normalized occupation measure until time $t$. Note that $L_t$ maps
\begin{equation}\label{map}
\Omega\to {\Mcal}_{1}(\R^d)
\end{equation}
inducing a probability distribution on ${\Mcal}_{1}(\R^d)$. Classical large deviation principle (\cite{DV}) states that
the family of these distributions satisfies a ``weak" large deviation principle in the space probability measures on $\Mcal_1(\R^d)$ equipped with the weak topology with a rate function $I$. More precisely, for every compact subset $K\subset\Mcal_1(\R^d)$,
\begin{equation}\label{DVub}
\limsup_{t\to\infty} \frac 1t \log \P(L_t\in K) \leq -\inf_{\mu\in K} I(\mu)
\end{equation}
and for every open subset $G\subset\Mcal_1(\R^d)$
\begin{equation}\label{DVlb}
\liminf_{t\to\infty} \frac 1t \log \P(L_t\in G) \geq -\inf_{\mu\in G} I(\mu),
\end{equation}
where 
$I$ is the rate function given by
\begin{equation}\label{Idef}
I(\mu) = 
\begin{cases}
\frac 12 \|\nabla f\|_2^2
\qquad\quad\mbox{if} \,\,f=\sqrt{\frac{\d \mu}{\d x}}\in H^1(\R^d)
\\
\infty \qquad\qquad\quad\quad\mbox{else.}
\end{cases}
\end{equation}
Here $H^1(\R^d)$ is the usual Sobolev space of square integrable functions with square integrable derivatives. Note that the function $\mu\mapsto I(\mu)$ is translation invariant and depends only on the orbit $\widetilde\mu$.
Furthermore, this map is convex and homogenous of degree $1$.

We say that a family of measures satisfies a ``strong" large deviation principle, or simply a {\it{large deviation principle}} (LDP) if the upper bound \eqref{DVub} holds for all closed sets.

Note that we also have an extension of \eqref{map} via
$$
\Omega\to {\Mcal}_{1}(\R^d)\to {\widetilde \Mcal}_{1}(\R^d)\subset \widetilde{\mathcal X}
$$ 
which induces a probability distribution $\mathbb Q_t$ of $\widetilde L_t$ on $\widetilde{\mathcal X}$. Our second main result gives a large deviation principle for $\mathbb Q_t$ on $\widetilde{\mathcal X}$ with the rate function
\begin{equation}\label{Itilde}
{\widetilde I}(\xi)=\sum_{\widetilde\alpha\in\xi} I({\widetilde\alpha}) 
\end{equation}
where
$$
I({\widetilde\alpha})=I(\alpha) 
$$
where $I$ is defined in \eqref{Idef} and $\alpha$ is any arbitrary element of the orbit $\widetilde\alpha$ (recall that $I$ is translation invariant). We remark that although $I$ is defined
in \eqref{Idef} only on probability measures $\Mcal_1(\R^d)$, the definition canonically extends to sub-probability measures $\Mcal_{\leq 1}(\R^d)$. Here is our second main result.
\begin{theorem}\label{ThmLDP}
The family of measures $\{\mathbb Q_t\}_t$ on the compact metric space ${\widetilde{\mathcal X}}$ equipped with the metric $\mathbf D$ satisfies a large deviation principle with the rate function ${\widetilde I}(\xi)$ defined in \eqref{Itilde}.
\end{theorem}
We split the proof into three main steps. First we prove that the function $\widetilde I$ is lower semicontinuous on $\widetilde{\mathcal X}$.

\begin{lemma}[{\bf Lower semicontinuity}] If $\xi_n\to \xi$ in ${\widetilde{\mathcal X}}$, then
$$
\liminf_{n\to\infty} {\widetilde I}(\xi_n)\ge {\widetilde I}(\xi).
$$
\end{lemma}
\begin{proof}
Let us first consider the case where, for each $n\in \N$, $\xi_n$ consists of a single orbit ${\widetilde\mu}_n$ and the limit
$\xi$ is a finite or countable collection $\{\widetilde\alpha_i\}$ 
arranged so that their masses $\{p_i\}$ form a non-increasing sequence. Given $\eps>0$, it is then possible to write (recall \eqref{mudecompose})
$$
\mu_n=\sum_{i=1}^k \alpha^{\ssup i}_n+\beta_n
$$
for some $k\in \N$ such that the following properties hold: For each $i=1,\dots,k$, there are sequences $\{a^{\ssup i}_n\}_n\subset \R^d$ such that 

\begin{gather*} 
\alpha^{\ssup i}_n\ast \delta_{a^{\ssup i}_n}\Rightarrow\alpha_i\in\widetilde\alpha_i,\\
\lim_{n\to\infty}\inf_{i\not=j}|a^{\ssup i}_n-a^{\ssup j}_n|=\infty,\\
 \lim_{n\to\infty}\int V(x-y)\alpha^{\ssup i}_n(\d x)\beta_n(\d y)=0,\\
 \limsup_{n\to\infty}\int V(x-y)\beta_n(\d x)\beta_n(\d y)\le 2\eps,
 \end{gather*}
for all $V\in \mathcal F_2$. In particular, since for each $i=1,\dots,k$, $\alpha^{\ssup i}_n$ is weakly convergent, they are a tight sequence and therefore $\alpha^{\ssup i}_n$ is concentrated near $-a^{\ssup i}_n$. We can find a smooth cut-off function $\varphi(x)$ which is $1$ in the unit ball, $0$ outside a ball of radius $2$ and smoothly varies in between. In particular, $0\le \varphi\le 1$. For $r_n>0$ to  be  suitably chosen later  we will have  a partition of unity by setting
 $$
 1=\sum_{i=1}^k\bigg\{\varphi\bigg(\frac{x+a^{\ssup i}_n}{r_n}\bigg)\bigg\}^2+\bigg[1-\sum_{i=1}^k\bigg\{\varphi\bigg(\frac{x+a^{\ssup i}_n}{r_n}\bigg)\bigg\}^2\bigg] \qquad\qquad r_n>0.
 $$
 We can assume that $I(\mu_n)<\infty$ for each $n\in \N$ (since otherwise there is nothing to prove) and hence $\mu_n(\d x)=f_n(x) \d x$ and $f_n\in H^1(\R^d)$. If $g_n=\sqrt{f_n}$ and $\frac{1}{2}\int_{\R^d}|\nabla g_n|^2 \d x\le \ell$,
  we need to prove that $\alpha_1,\ldots,\alpha_k$ are all absolutely continuous with densities $f^{\ssup 1},\dots,f^{\ssup k}$ and
 $$
 \sum_{i=1}^k  I(f^{\ssup i})\le \ell.
 $$
We define, for any $ i=1,\dots, k,$
$$
 \begin{aligned}
 f_n^{\ssup i}(x)&=f_n(x)\bigg\{\varphi\bigg(\frac{x+a^{\ssup i}_n}{r_n}\bigg)\bigg\}^2\\
 &=\bigg\{g_n(x)\varphi\bigg(\frac{x+a^{\ssup i}_n}{r_n}\bigg)\bigg\}^2
 \end{aligned}
 $$
and  we let $r_n\to\infty$ in such a way that $2r_n\le \min_{i\not=j}|a^{\ssup i}_n-a^{\ssup j}_n|$. Then $f^{\ssup i}_n(x)\d x\Rightarrow \alpha_i$ for $i=1,2,\ldots,k$ and
 $$
 I(f^{\ssup i}_n)=\frac{1}{2}\int \bigg|\nabla g_n(x)\varphi\bigg(\frac{x+a^{\ssup i}_n}{r_n}\bigg)+\frac{1}{r_n}g_n(x)\big(\nabla \varphi\big)\bigg(\frac{x+a^{\ssup i}_n}{r_n}\bigg)\bigg|^2\d x
 $$
 Since $r_n\to\infty$ , $\varphi$  and $\nabla\varphi$ are uniformly bounded and the integrals $\int |g_n(x)|^2 \d x$ and $\int |\nabla g_n|^2 \d x $  are bounded, only the first term in the integral counts. Since the functions 
 $$\bigg\{\varphi\bigg(\frac{x+a^{\ssup i}_n}{r_n}\bigg)\bigg\}_{i=1,\dots k}
 $$
  do not overlap and $0\le \varphi\le 1$, we infer
 $$
 \begin{aligned}
 \sum_{i=1}^k \frac{1}{2}\int\bigg\{ \big|\nabla g_n(x)\big|\varphi\bigg(\frac{x+a^{\ssup i}_n}{r_n}\bigg)\bigg\}^2\d x
 &\le \frac{1}{2}\int |\nabla g_n(x)|^2\d x
 \\
 &=\frac{1}{2}\int \big|\nabla \sqrt{f_n(x)}\big|^2\d x
 \\ 
& =  I(f_n)\\
&\le\ell.
\end{aligned}
 $$
 This implies that any weak limit $\alpha_i$ of $f^{\ssup i}_n \d x$ has a density $f^{\ssup i}$ and $\sum_{i=1}^k I(f^{\ssup i})\le \ell$.

 Finally if $\xi_n$ consists of multiple orbits $\{\xi^{\ssup i}_n\}_i$ with $\sum_i \widetilde I(\xi_n^{\ssup i})\le \ell$, we can choose subsequences such that, for each $i$, $\xi_n^{\ssup i}$ has a limit which is a collection $\xi^{\ssup i}$ of orbits $\{{\widetilde\alpha}_{j}^{\ssup i}\}_j$. The last step implies, for each $i$, $\sum_j {\widetilde I}({\widetilde \alpha}_{j}^{\ssup i})\le \ell^{\ssup i}$  where
 $\ell^{\ssup i}=\liminf_{n\to\infty} \widetilde I(\xi^{\ssup i}_n)$. Hence, 
 $$
 I(\xi)=\sum_i \ell^{\ssup i} \le \liminf_{n\to\infty} \sum_i\widetilde I(\xi^{\ssup i}_n) \leq \ell.
 $$
 This proves the lemma.
 \end{proof}
Next we derive the large deviation lower bound 
for $\mathbb Q_t$ on $\widetilde{\mathcal X}$. This is easily done given the translation invariance, convexity and homogeneity of $I$ and the denseness of the space $\widetilde{\Mcal}_1(\R^d)$ in $\widetilde{\mathcal X}$.

\begin{lemma}[{\bf Lower Bound}] \label{XLDPlb}
For any open set $G$ in $\widetilde{\mathcal X}$, 
\begin{equation}\label{thmlb}
\liminf_{t\to\infty}\frac{1}{t}\log \mathbb Q_t(G)\ge -\inf_{\xi\in G} {\widetilde I}(\xi)
\end{equation}
\end{lemma}
\begin{proof}
For \eqref{thmlb} it is enough to prove, given $\xi\in{\widetilde{\mathcal X}}$ with $\widetilde I(\xi)<\infty$, 
\begin{equation}\label{thmlb0}
\liminf_{t\to\infty}\frac{1}{t}\log \mathbb Q_t(U)\ge -{\widetilde I}(\xi).
\end{equation}
for any neighborhood $U\owns \xi$.


We claim that any $\xi\in{\widetilde{\mathcal X}}$ with ${\widetilde I}(\xi)<\infty$ can be approximated by $\xi_n\in \widetilde{\mathcal X}$ such that
\begin{equation}\label{claim}
\limsup_{n\to\infty} {\widetilde I}(\xi_n)\le {\widetilde I}(\xi).
\end{equation}
Indeed, recall from step-1 of the proof of Theorem \ref{T2.2} that $\widetilde\Mcal_1$ is dense in $\widetilde{\mathcal X}$ and $\xi= \{\widetilde\alpha_j\}\in \widetilde{\mathcal X}$ can be approximated 
by the sequence $(\widetilde\mu_n)_n$ in $\widetilde\Mcal_1$, where, as constructed in \eqref{convexcomb},
$$
 \mu_n:=  \sum_{j=1}^n \alpha_j\star  \delta_{a_j} + \bigg(1-\sum_{j=1}^n p_j\bigg) \lambda_M \in \Mcal_1(\R^d),
 $$
and $\lambda_M$ is a Gaussian with mean vector $\mathbf 0$ and covariance matrix $M\, \bf{Id}$. Furthermore, since $I(\cdot)$ on $\mathcal M_1(\R^d)$ is translation invariant, homogeneous of degree $1$ and convex, it is also sub-additive on $\mathcal M_{\le 1}(\R^d)$. Then,
$$
\begin{aligned}
I(\mu_n)&\le \sum_{j=1}^n I(\alpha_j\star \delta_{a_j})+\bigg(1-\sum_{j=1}^n p_j\bigg)I(\lambda_M)
\\
&= \sum_{j=1}^n I(\alpha_j)+\bigg(1-\sum_{j=1}^n p_j\bigg)I(\lambda_M)
\\
&\le {\widetilde I}(\xi)+I(\lambda_M)\\
&={\widetilde I}(\xi)+\frac{1}{M}.
\end{aligned}
$$
Since we can choose $M$ to depend on $n$, make it arbitrarily large and take $(\xi_n)$ to be the single orbit sequence $(\widetilde\mu_n)$,
\eqref{claim} is proved. The desired lower bound \eqref{thmlb0} now follows from the large deviation lower bound \eqref{DVlb} 
of the distribution of $L_t$ on $\Mcal_1(\R^d)$.
\end{proof}

Finally we turn to the large deviation upper bound for $\mathbb Q_t$. 

\begin{prop}[{{\bf{Upper bound of Theorem \ref{ThmLDP}}}}]\label{XLDPub}
For any closed set $F$ in $\widetilde{\mathcal X}$, 
\begin{equation}\label{thmub}
\limsup_{t\to\infty}\frac{1}{t}\log \mathbb Q_t(F)\leq -\inf_{\xi\in F} {\widetilde I}(\xi)
\end{equation}
\end{prop}
Let  $\mathcal U$ be the space of  functions of the form $u=c+v$  where $v$ is a smooth nonnegative function  with compact support on $\R^d$ and $c>0$ is a positive constant. Let $\varphi(x)$ be a smooth function satisfying $0\le \varphi(x)\le 1$, $\varphi(x)=1$ inside the unit ball and $\varphi(x)=0$ outside the ball of radius $2$. For  any $k\ge 1, R >0$, $u_1,\ldots, u_k\in \mathcal U$ and $a_1,\ldots, a_k\in\R^d$ and  $c>0 $ consider the function
\begin{equation}\label{gdef}
g(x)= g(k, R, c, a_1,\ldots, a_k, x)=c+ \sum_{i=1}^k u_i(x+a_i) \varphi\bigg(\frac{x+a_i}{R}\bigg) 
\end{equation}
and define $F: \Omega\to \R$ by setting

\begin{equation}\label{F}
\begin{aligned}
F(u_1,\ldots, u_k, c,  R, t, \omega)&=\sup_{a_1,\ldots a_k\atop \inf_{i\not=j}|a_i-a_j|\ge 4R} \frac{1}{t}\int_0^t  \frac{-\frac{1}{2}\Delta g\big(W(s)\big)}{g\big(W(s)\big)}ds\\
&=\sup_{a_1,\ldots a_k\atop \inf_{i\not=j}|a_i-a_j|\ge 4R} \int_{\R^d}  \frac{-\frac{1}{2}\Delta g( x)}{g(x)}L_t(\d x).
\end{aligned}
\end{equation}
Since the last expression depends only on the image $\widetilde{L}_t$ of $L_t$ in $\widetilde{\mathcal X}$, we write
\begin{equation}\label{Ftilde}
\begin{aligned}
\widetilde {F} \big(u_1,\ldots, u_k, c,R,  {\widetilde L}_t\big)&=\sup_{a_1,\ldots a_k\atop \inf_{i\not=j}|a_i-a_j|\ge 4R} \int_{\R^d}  \frac{-\frac{1}{2}\Delta g( x)}{g(x)}L_t(\d x)\\
&=F(u_1,\ldots, u_k, c,  R, t, \omega).
\end{aligned}
\end{equation}
 We will need the next three lemmas to prove the upper bound. First we prove that
$\widetilde F(\cdot)$ grows only sub-exponentially as $t\to\infty$. 

\begin{lemma}\label{L3.2}
For  any $k\ge 1, R >0$, $u_1,\ldots, u_k\in \mathcal U$ and $c>0$,
\begin{equation}\label{subexp}
\begin{aligned}
\limsup_{t\to\infty}&\frac{1}{t} \log \E\bigg\{\exp\big\{t \widetilde{F}(u_1,\ldots,u_k,c, R, \widetilde{L}_t)\big\}\bigg\}\\
&=\limsup_{t\to\infty}\frac{1}{t} \log \E\bigg\{\exp\big\{t F(u_1,\ldots, u_k, c,  R, t, \omega)\big\}\bigg\} \\
&\leq 0.
\end{aligned}
\end{equation}
\end{lemma}
\begin{proof} If it were not for the supremum over $a_1,\ldots, a_k$ this would be a simple consequence of Feynman-Kac formula. In fact, we first show that, 
\begin{equation}\label{FK}
\limsup_{t\to\infty}\frac{1}{t} \log \E\bigg\{\exp\bigg\{\int_0^t  \frac{-\frac{1}{2}\Delta g\big(W(s)\big)}{g\big(W(s)\big)}ds\bigg\}\bigg\}=0
\end{equation}
Indeed, by the Feynman-Kac formula,
 the function
$$
 \Psi(t,x)= \E_x\bigg\{ g(W_t)\exp\bigg\{ \int_0^t \frac {\Delta g(W_s)}{2u(W_s)}\bigg\}\bigg\}
 $$
 satisfies the initial value problem
$$
 \begin{cases}
 \frac{\partial}{\partial t} \Psi= -\frac  12 \Delta \Psi (t,x) + \frac {\Delta g(x)}{2g(x)} \Psi (t,x) \\
  \Psi(0,x)=  g(x).
 \end{cases}
 $$
 However, we clearly see that 
 $$
 \Psi(t,x)= g(x)
 $$
 solves the above heat equation. Furthermore by definition (recall \eqref{gdef}),
 $$
 g(x) \geq c.
 $$   
 Hence, we conclude,
 $$
 \begin{aligned}
 g( x) &= \E_x\bigg\{ g(W_t)\exp\bigg\{ \int_0^t \frac {\Delta g(W_s)}{2g(W_s)}\bigg\}\bigg\}  \\
 &\geq c\, \E_x\bigg\{ \exp\bigg\{ \int_0^t \frac {\Delta g(W_s)}{2g(W_s)}\bigg\}\bigg\} 
 \end{aligned}
 $$
 and therefore,
 \begin{equation}\label{subsexpest1}
 \E_x\bigg\{ \exp\bigg\{ \int_0^t \frac {\Delta g( W_s)}{2g(W_s)}\bigg\}\bigg\} 
 \leq \frac{g(x)}c.
 \end{equation} 
 This proves \eqref{FK}.
To handle the supremum over $(a_1,\dots a_k)$ inside the expectation we have to do  a ``course graining" argument. 

First we note that if the range of the Brownian motion in the time interval $[0,t]$ is $r_t$, once any $|a_i|$ exceeds $r_t+R$ it will no longer affect the value of $g$ (again recall the definition \eqref{gdef}). We can therefore limit each $a_i$ to the ball of radius $r_t+R$. But $P[r_t+R\ge  t^2]\le \exp[-c_1t^3]$ and can be ignored. In other words, we can limit each $a_i$ to the ball of radius $t^2$.

Furthermore, the function $\frac{-\frac{1}{2}\Delta g (x)}{g(x)}$ is a uniformly continuous function of $a_1,\ldots, a_k$ and given any $\eps>0$, there is a $\delta>0$ such that  its oscillation in a box of size
$\delta$ is at most $\eps$. A ball of radius $t^2$ can be covered by $\big(\frac{t^2}{\delta}\big)^{dk}$ such boxes.
There is a set $K\subset (\R^d)^k$ of representatives  of such boxes, a set of cardinality at most  $\big(\frac{t^2}{\delta}\big)^{dk}$  satisfying $|a_i-a_j|\ge 4R$ for all $i\not=j$. 

Using the above two remarks, we can now estimate,
\begin{align*}
\E\bigg\{&\exp\bigg\{\sup_{a_1,\ldots,a_k\atop |a_i-a_j|\ge 4R \forall i\not=j}\int_0^t \frac{-\frac{1}{2}\Delta g(W_s)}{g(W_s)}ds\bigg\}\bigg\}\\
&\le \E\bigg\{\exp\bigg\{\sup_{|a_1|\le t^2,\ldots,|a_k|\le t^2\atop |a_i-a_j|\ge 4R \forall i\not=j}\int_0^t \frac{-\frac{1}{2}\Delta g(W_s)}{g(W_s)}ds\bigg\}\bigg\}+e^{c_2t}P\bigg(\sup_{0\le s\le t} |W_s|\ge t^2\bigg) \\
&\le \E\bigg\{\exp\bigg\{\eps t+\sup_{(a_1,\ldots,a_k)\in K}\int_0^t \frac{-\frac{1}{2}\Delta g(W_s)}{g(W_s)}ds\bigg\}\bigg\}+\exp\big(c_2t-c_1t^3\big)\\
&\le \E\bigg\{\sum_{(a_1,\ldots,a_k)\in K}\exp\bigg\{\eps t+\int_0^t \frac{-\frac{1}{2}\Delta g(W_s)}{g(W_s)}ds\bigg\}\bigg\} +\exp\big(c_2t-c_1t^3\big)\\
&\le   \bigg(\frac{t^2}{\delta}\bigg)^{dk} \sup_{(a_1,\ldots,a_k)\in K}E\bigg\{\exp\bigg\{\eps t+\int_0^t \frac{-\frac{1}{2}\Delta g(W_s)}{g(W_s)}ds\bigg\}\bigg\}+\exp\big(c_2t-c_1t^3\big)
\end{align*}
Taking logarithm, dividing by $t$, passing to $\lim_{t\to\infty}$ and invoking \eqref{FK}, we obtain

$$
\limsup_{t\to\infty}\frac{1}{t}\log E\bigg\{\exp\bigg\{\sup_{a_1,\ldots,a_k\atop |a_i-a_j|\ge 4R \forall i\not=j}\int_0^t \frac{-\frac{1}{2}\Delta g(W_s)}{g(W_s)}ds\bigg\}\bigg\}\le \eps,
$$
and $\eps>0$ is arbitrary. \eqref{subexp} is proved.
\end{proof}

\begin{lemma}\label{L3.3}
Let $({\widetilde\mu}_n)_n$ be sequence in $\widetilde {\mathcal X}$ which converges to $\xi=\{\widetilde{\alpha}_j\}\in\widetilde{\mathcal X}$. For any $k\in \N$, $i=1,\dots,k$ and
 $u_{i,\ell}(x)=u_i(x)\varphi(\frac{x}{R})$, where $u_i\in \mathcal U$,  we have
\begin{equation}\label{Lambdadef}
\begin{aligned}
\liminf_{n\to\infty} \widetilde {F}(u_1,\ldots, u_k, c,R,  {\widetilde \mu}_n)&\ge \sum_{i=1}^k  \int \frac{-(\frac{1}{2}\Delta u_{i,R})(x)}{ c+u_{i,R}(x)} \alpha_i(\d x)\\
&=\widetilde\Lambda (\xi,R,  c,  u_1,\ldots,u_k).
\end{aligned}
\end{equation}
\end{lemma}
\begin{proof}
If $\widetilde\mu_n\to\xi=\{{\widetilde \alpha}_j\}$, then for $j=1,2,\ldots k$ we can again decompose $\mu_n$ as (recall \eqref{mudecompose}) $\mu_n=\sum_{j=1}^k \alpha_{n,j}+\beta_n$  with $\alpha_{n,j}\ast\delta_{a_{n,j}}\Rightarrow \alpha_j$ for a suitable choice of $a_{n,j}$ that satisfy
$\lim_{n\to\infty}|a_{n,i}-a_{n,j}|=\infty$ for $i\not= j$. If $n$ is large enough $|a_{n,i}-a_{n,j}|\ge 4R$ and
the supports of $\{u_{i,R}\}_i$ are mutually disjoint. In particular with the choice of $a_i=-a_{n,i}$
$$
 \frac{-\frac{1}{2}\Delta g(k, R,c,   a_1,\ldots, a_k, x)}{g(k, R, c,  a_1,\ldots, a_k, x)}=\sum_{i=1}^k\frac{-(\frac{1}{2}\Delta u_{i,R})(x-a_{n,i})}{ c+u_{i,R}(x-a_{n,i})} 
$$
Because $\alpha_{n,j}$ is gets widely separated from $\beta_n$ as well as $\alpha_{n,i}$ for $i\not= j$, it is clear that
$$
\lim_{n\to\infty}\int \frac{-(\frac{1}{2}\Delta u_{i,R})(x-a_{n,i})}{ c+u_{i,R}(x-a_{n,i})} \mu_n(\d x)=\int \frac{-(\frac{1}{2}\Delta u_{i,R})(x)}{ c+u_{i,R}(x)} \alpha_i(\d x),
$$
and the lemma follows.
\end{proof}
\begin{lemma}\label{L3.4}
With $\widetilde\Lambda$ defined in \eqref{Lambdadef} and $\widetilde I$ defined in \eqref{Itilde}, we have the identification 
$$
\widetilde {I}(\xi)=\sup_{R, c>0, k\in \N, \atop  u_1, \ldots, u_k\in\mathcal U} \widetilde\Lambda (\xi, R, c, u_1,\ldots,u_k).
$$
\end{lemma}
\begin{proof}
Recall the definition of the classical rate function $I$ from \eqref{Idef}. For any $\alpha\in{\mathcal M}_{\le 1}(\R^d)$, we can also identify $I$ as
$$
I(\alpha)=\sup_{u\in \mathcal {U}\atop c>0}\int \frac{-\frac{1}{2}\Delta u (x)}{c+u(x)}\alpha(\d x).
$$
Therefore for every $k\in \N$,
$$
\sup_{c >0,R>0\atop u_1, \ldots, u_k\in\mathcal U} \widetilde\Lambda (\xi, c, R, u_1,\ldots,u_k)=\sum_{i=1}^k I(\alpha_j)
$$
and 
$$
\sup_{k\in \N} \sum_{j=1}^k I(\alpha_j)=\sum_{j=1}^\infty I(\alpha_j)=\widetilde{I}(\xi).
$$
\end{proof}
Now we come to the proof of the large deviation upper bound for $\mathbb Q_t$ in $\widetilde{\mathcal X}$.

{\bf{Proof of Proposition \ref{XLDPub}}:} $\widetilde{\mathcal X}$ being compact, for \eqref{thmub}, it is enough to prove (by the usual machinery of covering a compact space by finitely many balls and invoking the union of events bound) that if $\xi\in\widetilde{\mathcal X}$ and $N_\delta$ is a ball (as usual, in the metric $\mathbf D$) of radius $\delta$ around $\xi$, then 
\begin{equation}\label{localub}
\limsup_{\delta\to 0}\limsup_{t\to\infty}\frac{1}{t}\log \mathbb Q_t(N_\delta)\le -\widetilde{I}(\xi).
\end{equation}
Let ${\mathcal H}$ be the space of  maps $ H: \widetilde{\mathcal M}_1(\R^d)\to\R$ with the following properties: For each $ H$ there is a corresponding function $\Lambda_{H}: \mathcal {\widetilde X}\to \R$ such that 
\begin{equation}\label{limitH}
\liminf_{ {\widetilde \mu}\in \widetilde {\mathcal  M}_1(\R^d)\atop {\widetilde\mu}\to\xi\in{\widetilde{\mathcal X}}}H({\widetilde \mu})\ge \Lambda_{ H}(\xi)
\end{equation}
and
\begin{equation}\label{Hsubexp}
\limsup_{t\to\infty}\frac{1}{t}\log \E^{\mathbb Q_t}\big\{\exp\{ t H(\cdot)\}\big\}\leq 0.
\end{equation}
Then again the properties of the decomposition \eqref{mudecompose}, 
a routine application of Tchebycheff's inequality, \eqref{limitH} and \eqref{Hsubexp} show that, for any $ H\in{\mathcal{H}}$,
$$
\limsup_{\delta\to 0}\limsup_{t\to\infty}\frac{1}{t}\log \mathbb Q_t[N_\delta]\le -\Lambda_{ H}(\xi).
$$
It is therefore enough to identify $\widetilde{I}(\xi)$ as
\begin{equation}\label{ILambda}
\widetilde I(\xi)=\sup_{H\in {\mathcal H}} \Lambda_{H}(\xi).
\end{equation}
Recall the definition of $\widetilde F$ from \eqref{Ftilde}. Then, for $\mathcal H$, by Lemma \ref{L3.2}, Lemma \ref{L3.3} and Lemma \ref{L3.4}, we can take the collection $\{\widetilde{F}(u_1,\ldots, u_k, c,  R, \widetilde{\mu})\}$ with $k\in \N, R, c>0, u_1,\ldots,u_k\in\mathcal {U}$ and $\mu\in \Mcal_1(\R^d)$ and set $\Lambda_{\widetilde{F}}=\widetilde\Lambda$, with $\widetilde\Lambda$ defined in \eqref{Lambdadef}. This proves \eqref{ILambda} and hence Proposition \ref{XLDPub}.
\qed

\section{Application: Localization of path measures with Coulomb interaction}\label{application}

In this section we come back to the problem we introduced in Section \ref{Introduction}. 
Again we consider the Wiener measure $\P$ on $\Omega= C_0([0,\infty); \R^3)$ corresponding to a three dimensional Brownian motion $W=(W_t)_{t\geq 0}$ starting at the origin. Consider the transformed measure
$$
\d{\widehat \P}_t=\frac{1}{Z_t}\exp\bigg\{\frac{1}{t}\int_0^t\int_0^t \frac{1}{|W_s-W_\sigma|} \d s \d\sigma\bigg\} \d \P
$$
where 
$$
 Z_t=\E\bigg[ \exp\bigg\{\frac{1}{t}\int_0^t\int_0^t \frac{1}{|W_s-W_\sigma|} \d s \d\sigma\bigg\}\bigg]
$$
is the normalizing constant or the {\it{partition function}}. As mentioned before (see \cite{DVP}),
\begin{equation}\label{varfor}
\lim_{t\to\infty}\frac 1t\log Z_t
=\sup_{\heap{\psi\in H^1(\R^3)}{\|\psi\|_2=1}} 
\Bigg\{\int_{\R^d}\int_{\R^d}\d x\d y\,\frac{ {\psi^2(x) \psi^2(y)}}{|x-y|}-\frac 12\big\|\nabla \psi\big\|_2^2\Bigg\},
\end{equation}
 and according to the classical result of Lieb (see \cite{L}), this variational problem
admits a maximizer $\psi_0$ which is radially symmetric and is unique up to translations. Let $d\mu_0=\psi_0^2(x) \d x$ define its probability distribution with $\widetilde{\mu}_0$ the corresponding orbit 
in $\widetilde {\mathcal M}_1(\R^3)$.  We study the distribution 
$${\widehat {\mathbb Q}}_t=\widehat {\P}_t \widetilde L_t^{-1}
$$ 
on $\widetilde {\mathcal M}_1(\R^3)$ of the orbit ${\widetilde L}_t$ of the normalized 
occupation measures $L_t$  of the  trajectory $\{W_s:0\le s\le t\}$ under the transformed measure $\widehat {\P}_t$.
Here is our next main result.
\begin{theorem}[The tube property under Coulomb interaction]\label{T4.1} 
As probability mesures on $\widetilde {\mathcal M}_1(\R^3)$,
$$
\lim_{t\to\infty} \widehat {\mathbb Q}_t=\delta_{\widetilde{\mu}_0}
$$
under the weak topology.
\end{theorem}
\begin{remark}
Note that the topology on $\widetilde {\mathcal M}_1(\R^3)$ is the same as weak convergence. As we shall see, the compactification $\widetilde{\mathcal X}$ of $\widetilde {\mathcal M}_1(\R^3)$ plays a role only in the proof of this theorem and not in its statement. 
\end{remark}
The proof involves  the standard large deviation route. The function $\frac{1}{|x|}$ is unbounded and needs to be truncated to fit within  the standard large deviation theory. We write 
\begin{equation}\label{coulomb}
\frac{1}{|x|}= V_\eps(x)+Y_\eps(x)
\end{equation} 
with $V_\eps (x)=(\eps^2+|x|^2)^{-\frac{1}{2}}$. The difference 
is given by 
$$
\begin{aligned}Y_\eps(x)=\frac{1}{|x|}-\frac{1}{\sqrt {\eps^2+|x|^2}}
&=\frac{\sqrt {\eps^2+|x|^2}-|x|}{|x|\sqrt{\eps^2+|x|^2}}\\
&=\frac{\eps^2}{|x|+\sqrt {\eps^2+|x|^2}}\,\,\frac{1}{\sqrt {\eps^2+|x|^2}}\,\,\frac{1}{|x|}\\
&=\eps^{-1}\varphi\big(\frac{x}{\eps}\big),
\end{aligned}$$
with
$$
\phi(x)=\frac{1}{|x|}\,\,\frac 1{\sqrt{1+|x|^2}}\,\, \frac 1{\big(|x|+\sqrt{1+|x|^2\big)}}$$
We need the following lemma to control the difference.
\begin{lemma}\label{L4.2}
For any $\lambda>0$,
\begin{equation}\label{superexp}
\limsup_{\eps\to 0}\limsup_{t\to\infty}\frac{1}{t}\log \E\bigg[ \exp \frac{\lambda}{t}\int_0^t\int_0^t Y_\eps(W_s-W_\sigma)\d s\d\sigma\bigg]=0.
\end{equation}
\end{lemma}

\begin{proof}
One can bound  $\phi(x)$ which behaves like $\frac{1}{|x|}$ near $0$ and like $\frac{1}{|x|^3}$ near $\infty$ by $\frac{C}{|x|^\frac{3}{2}}$. In particular 
$$
Y_\eps(x)\le \frac{C\sqrt{\eps}}{|x|^\frac{3}{2}}.
$$
Then by time ordering and Jensen's inequality,
$$
\begin{aligned}
&\exp \bigg\{\frac{\lambda}{t}\int_0^t\int_0^t Y_\eps(W_s-W_\sigma)\d s\d\sigma\bigg\}\\
&=\exp \bigg\{\frac{2\lambda}{t}\int_0^t\bigg\{\int_s^t Y_\eps(W_s-W_\sigma)\d\sigma\bigg\}\d s\bigg\}\\
&\le \frac{1}{t}\int_0^t   \exp\bigg\{2\lambda \int_s^T Y_\eps(W_s-W_\sigma)\d\sigma\bigg\} \d s\\
&\le  \frac{1}{t}\int_0^t   \exp\bigg\{2C \lambda\sqrt{\eps} \int_s^t \frac{1}{|W_s-W_\sigma|^\frac{3}{2}}\d \sigma\bigg\} \d s \\
&=^{\ssup{\mathcal D}} \frac 1t\int_0^t \d s\,\exp\bigg\{2C\lambda\sqrt \eps \int_0^{t-s} \frac 1{|W_\sigma|^{3/2}} \d \sigma\bigg\}
\end{aligned}
$$
If we can show that, for $\eps>0$ small enough, 
\begin{equation}\label{check1}
\sup_{x\in\R^3} \E^{\ssup x} \bigg\{\exp\bigg\{2C\lambda\sqrt\eps\int_0^1 \frac{1}{|W_\sigma|^{3/2}} \d \sigma\bigg\}\bigg\}\leq \alpha<\infty,
\end{equation}
then it follows, by successive conditioning and the Markov property,
\begin{equation}\label{check2}
 \E \bigg\{\exp\bigg\{2C\lambda\sqrt\eps\int_0^{t-s} \frac{1}{|W_\sigma|^{3/2}} \d \sigma\bigg\}\bigg\}\leq \alpha^{t-s}.
 \end{equation}
This will prove \eqref{superexp}.

It remains to check \eqref{check1}. For this, we appeal to Portenko's lemma (see \cite{P}), which states that, if for a Markov process $\{\P^{\ssup x}\}$ and for a function $\widetilde V\ge 0$
$$
\sup_{x\in \R^d} \E^{\ssup x}\bigg\{\int_0^1 \widetilde V(W_s)\d s\bigg\}\le \eta<1
$$
then
$$
\sup_{x\in \R^d} \E^{\ssup x}\bigg\{\exp\bigg\{\int_0^1 \widetilde V(W_s)\d s\bigg\}\bigg\}\le \frac{\eta}{1-\eta}=\alpha <\infty.
$$
Hence, to prove \eqref{check1}, we need to verify that 
\begin{equation}\label{check3}
\begin{aligned}
\sup_{x\in \R^3} \E^{\ssup x}\bigg\{\int_0^1 \frac{\d \sigma}{|W_\sigma|^\frac{3}{2}}\bigg\}
=\sup_{x\in \R^3} \d y\int_0^1 \d \sigma\int_{\R^3}\frac{1}{|y|^\frac{3}{2}} \frac{1}{(2\pi \sigma)^\frac{3}{2} }\exp\bigg\{-\frac{(y-x)^2}{2\sigma}\bigg\}
<\infty.
\end{aligned}
\end{equation}
One can see that  
$$
\sup_{x\in \R^3} \int_{\R^3}\d y\frac{1}{|y|^\frac{3}{2}} \frac{1}{(2\pi \sigma)^\frac{3}{2} }\exp\bigg\{-\frac{(y-x)^2}{2\sigma}\bigg\}
$$ 
is attained at $x=0$ because we  can rewrite the integral by Parseval's identity as 
$$ c\int_{\R^3} \exp\bigg\{-\frac{\sigma|\xi|^2}{2}+i\langle x,\xi\rangle\bigg\}\frac{1}{|\xi|^\frac{3}{2}}d\xi,
$$ 
where $c>0$ is a constant. When  $x=0$ the integral reduces to $\int_0^1 \sigma^{-3/4} \,\,\d \sigma$ which is finite.
\end{proof}
We continue with the proof of the Theorem \ref{T4.1}. First we prove
a large deviation estimate for  $\widehat {\mathbb Q}_t$.
\begin{theorem}
For any closed set $F\subset \widetilde{\mathcal X}$
\begin{equation*}
\limsup_{t\to\infty}\frac{1}{t}\log \widehat {\mathbb Q}_t[F]\le -\inf_{\xi\in F} \widetilde {J}(\xi),
\end{equation*}
and for any open set $G\subset\widetilde{\mathcal X}$
\begin{equation*}
\liminf_{t\to\infty}\frac{1}{t}\log \widehat {\mathbb Q}_t[G]\ge -\inf_{\xi\in G} \widetilde {J}(\xi),
\end{equation*}
where, for $\xi=\{\widetilde{\alpha}_j\}\in \widetilde{\mathcal X}$,
$$
{\widetilde J}(\xi)= \widetilde\rho-\sum_j \bigg\{\int \frac{1}{|x-y|} \alpha_j(\d x)\alpha_j(\d y)-{\widetilde I}(\widetilde{\alpha}_j)\bigg\}
$$
and  $\widetilde\rho$ is given by
$$
\widetilde\rho=\sup_{\xi\in\widetilde{\mathcal X}} \sum_j \bigg\{\int_{\R^3}\int_{\R^3} \frac{\psi_j^2(x)\psi_j^2(y)}{|x-y|} \d x\d y-\frac{1}{2}\sum_j\big\|\nabla \psi_j\big\|_2^2\bigg\}
$$ 
 and $\alpha_j(\d x)=\psi^2_j(x)\d x$ with $\sum_j \int_{\R^3} \psi_j^2(x) \d x\leq 1$.
\end{theorem}
\begin{proof}
We fix a closed set $F\subset \widetilde{\mathcal X}$. Then, by definition,
\begin{equation}\label{rewrite}
\begin{aligned}
\widehat {\mathbb Q}_t(F)&= \widehat{\P}_t( \widetilde L_t\in F)\\
&= \frac{\E^{\mathbb Q_t}\bigg\{ \exp\big\{\frac 1t\int_0^t\int_0^t \frac 1{|W_\sigma-W_s|} \d\sigma \d s \big\} \, \1_F\bigg\}}{\E^{\mathbb Q_t}\bigg\{ \exp\big\{\frac 1t\int_0^t\int_0^t \frac 1{|W_\sigma-W_s|} \d\sigma \d s \big\}\bigg\}},
\end{aligned}
\end{equation}
where $\mathbb Q_t$ is the distribution of $\widetilde L_t$ in $\widetilde\Mcal_1(\R^d)$. We first handle the numerator. The denominator will be taken care of similarly. Recall the decomposition \eqref{coulomb}. Then 
with $\frac{1}{p}+\frac{1}{q}=1$ and H\"older's inequality, the numerator becomes 
\begin{align*}
\int_F& \exp\bigg[\frac{1}{t}\int_0^t\int_0^t \bigg\{V_\eps(|W_s-W_\sigma|)+Y_\eps(|W_s-W_\sigma|)\bigg\}\d\sigma\,\d s\bigg] d\mathbb Q_t\\
&\le\bigg[ \int_F \exp\bigg\{\frac{1}{t}\int_0^t\int_0^t p\,V_\eps(|W_s-W_\sigma|) \d\sigma\,\d s\bigg\}  d\mathbb Q_t\bigg]^\frac{1}{p}\\
&\qquad\times \bigg[\int_F \exp\bigg\{\frac{1}{t}\int_0^t\int_0^t  q\,Y_\eps(|W_s-W_\sigma|)]\d\sigma\,\d s\bigg\} d\mathbb Q_t\bigg]^\frac{1}{q}\\
\end{align*}
Taking logarithm, dividing by $t$, passing to $\limsup_{t\to\infty}$ and followed by $\eps\to 0$,
\begin{align*}
\limsup_{t\to\infty}&\frac{1}{t}\log \int_F \exp\bigg[\frac{1}{t}\int_0^t\int_0^t \frac{1}{|W_s-W_\sigma|}\d\sigma\,\d s\bigg] \d\mathbb Q_t\\
&\le \limsup_{\eps\to 0}\frac{1}{p}\limsup_{t\to\infty}\frac{1}{t}\log \int_F \exp\bigg[\frac{1}{t}\int_0^t\int_0^t p\,V_\eps(|W_s-W_\sigma|)\d\sigma\,\d s\bigg] \d\mathbb Q_t\\
&\qquad +\limsup_{\eps\to 0}\frac{1}{q}\limsup_{t\to\infty}\frac{1}{t}\log \int_F \exp\bigg[\frac{1}{t}\int_0^t\int_0^t  q\,Y_\eps(|W_s-W_\sigma|)\d\sigma\,\d s\bigg] \d\mathbb Q_t.
\end{align*}
By Lemma \ref{L4.2} the second term is $0$. For the first term, since for every $\eps>0$, $V_\eps \in \mathcal F_2$, by Corollary \ref{continuity}, Proposition \ref{XLDPub} and Varadhan's lemma, 
$$
\begin{aligned}
\limsup_{t\to\infty}\frac{1}{t}\log &\int_F \exp\bigg[\frac{1}{t}\int_0^t\int_0^t p\,V_\eps(|W_s-W_\sigma|)\d\sigma\,\d s\bigg] \d\mathbb Q_t \\
&\leq \sup_{\xi\in F}\bigg[\sum_j \int_{\R^3}\int_{\R^3} pV_\eps (x-y)\psi^2_j(x)\psi^2_j(y) \d x\d y-\frac{1}{2}\sum_j \|\nabla \psi_j\|^2 \bigg],
\end{aligned}
$$
where $\xi=\{\widetilde\alpha_j\}$ and $\alpha_j(\d x)= \psi_j^2(x)\d x$ with $\sum_j \int_{\R^3} \psi_j^2(x)\d x\leq 1$.

Since  $V_\eps(x)\le \frac{1}{|x|}$ and $V_\eps(x)\to \frac{1}{|x|}$ as $\eps\to 0$, for any $p>1$,
\begin{align*}
\lim_{\eps\to 0}&\sup_{\xi\in F}\bigg[\sum_j \int_{\R^3}\int_{\R^3} pV_\eps (x-y)\psi^2_j(x)\psi^2_j(y) \d x\d y-\frac{1}{2}\sum_j \|\nabla \psi_j\|^2 \bigg]\\
&=\sup_{\xi\in F}\bigg[\sum_j \int_{\R^3}\int_{\R^3} p\,\frac{1}{|x-y|}\psi^2_j(x)\psi^2_j(y) \d x\d y-\frac{1}{2}\sum_j\|\nabla \psi_j\|^2\bigg].
\end{align*}
We can  now let $p\to 1$ and obtain
\begin{equation}\label{numub}
\begin{aligned}
\limsup_{t\to\infty}&\frac{1}{t}\log \int_F \exp\bigg[\frac{1}{t}\int_0^t\int_0^t \frac{1}{|W_s-W_\sigma|}\d\sigma\,\d s\bigg] \d\mathbb Q_t\\
&\le\sup_{\xi\in F}\bigg[\sum_j \int_{\R^3}\int_{\R^3} \,\frac{1}{|x-y|}\psi^2_j(x)\psi^2_j(y) \d x\d y-\frac{1}{2}\sum_j\|\nabla \psi_j\|^2\bigg].
\end{aligned}
\end{equation}
 The lower bound
 \begin{equation}\label{numlb}
\begin{aligned}
\limsup_{t\to\infty}&\frac{1}{t}\log \int_G \exp\bigg[\frac{1}{t}\int_0^t\int_0^t \frac{1}{|W_s-W_\sigma|}\d\sigma\,\d s\bigg] \d\mathbb Q_t\\
&\ge \sup_{\xi\in G}\bigg[\sum_j \int_{\R^3}\int_{\R^3} \,\frac{1}{|x-y|}\psi^2_j(x)\psi^2_j(y) \d x\d y-\frac{1}{2}\sum_j\|\nabla \psi_j\|^2\bigg],
\end{aligned}
\end{equation}
for open sets $G\subset \widetilde{\mathcal X}$ follows immediately from Lemma \ref{XLDPlb}. This derives the asymptotic behavior of the numerator in \eqref{rewrite}.
For the denominator, we invoke \eqref{numub} for $F=\widetilde{\mathcal X}$ and \eqref{numlb} for $G=\widetilde{\mathcal X}$ to deduce 
\begin{align}
\lim_{t\to\infty}&\frac{1}{t}\log \int_{\widetilde{\mathcal X}}  \exp\bigg[\frac{1}{t}\int_0^t\int_0^t \frac{1}{|W_s-W_\sigma|}\d\sigma\,\d s\bigg] \d\mathbb Q_t\notag\\
&= \sup_{\xi\in \widetilde{\mathcal X}}\bigg[\sum_j \int_{\R^3}\int_{\R^3} \,\frac{1}{|x-y|}\psi^2_j(x)\psi^2_j(y) \d x\d y-\frac{1}{2}\sum_j\|\nabla \psi_j\|^2\bigg]\label{eq4.1}\\
&=\widetilde\rho\notag.
\end{align}
We apply \eqref{numub}, \eqref{numlb} and \eqref{eq4.1} to \eqref{rewrite}. The theorem is proved.
\end{proof}
We need a lemma here to complete the proof of  Theorem \ref{T4.1}.
\begin{lemma}
The supremum in \eqref{eq4.1} is attained  only when $\xi$ consists of a single orbit $\widetilde{\mu}$ with $\mu(\d x)=\psi^2(x)\d x$  for a unique radially symmetric $\psi$ and $\int_{\R^3} \psi(x)^2 \d x=1$.
\end{lemma}
\begin{proof} If we rescale with $\psi(x)$ being replaced by $\sigma^2 \psi(\sigma x)$, the expression
$$\sigma^8\int_{\R^3}\int_{\R^3} \frac{1}{|x-y|}\psi^2(\sigma x)\psi^2(\sigma y) \d x\d y-\sigma^6 \frac{1}{2}\int_{\R^3}|\nabla \psi(\sigma x)|^2 \d x$$
becomes
$$
\sigma^3\int_{\R^3}\int_{\R^3} \frac{1}{|x-y|}\psi^2(x)\psi^2(y) \d x\d y-\frac{1}{2}\sigma^3\int|\nabla \psi(x)|^2 \d x
$$
while the mass $\sigma^4 \int_{\R^3}  \psi^2(\sigma x) \d x$  becomes  $\sigma  \int_{\R^3}  \psi^2(x) \d x$. Therefore if we define

$$
\rho(m)=\sup_{\int_{\R^3} h^2(x)\d x=m}\bigg[ \int_{\R^3}\int_{\R^3} \frac{1}{|x-y|}\psi^2(x)\psi^2(y) \d x\d y-\frac{1}{2}\int_{\R^3} |\nabla \psi(x)|^2 \d x\bigg]$$
then $\rho(m)=Cm^3$. In particular $\rho(m_1+m_2)>\rho(m_1)+\rho(m_2)$ proving that supremum in \eqref{eq4.1} is attained at a single orbit $\xi=\{\widetilde\mu\}$ of total mass $\mu(\R^3)=1$. According to Lieb's theorem (see
\cite{L}), the function $\psi$ that maximizes 
$$
\bigg[ \int_{\R^3}\int_{\R^3} \frac{1}{|x-y|}\psi^2(x)\psi^2(y) \d x\d y-\frac{1}{2}\int_{\R^3} |\nabla \psi(x)|^2 \d x\bigg]$$
subject to  $\int_{\R^3} \psi^2(x)\d x=1$ is unique up to translation.
\end{proof}

{\bf{Acknowledgement.}} The first author would like to thank Erwin Bolthausen (Zurich) and Wolfgang Koenig (Berlin) for many helpful discussions. The second author was supported
partially by NSF grant DMS 1208334. Both authors would like to thank three anonymous referees whose input led to a more elaborate version of the present manuscript.


\begin{thebibliography}
\bibliographystyle{}
\smallskip
 \bibitem{BS97}
{\sc E. Bolthausen} and {\sc U. Schmock},
\newblock On self-attracting $d$-dimensional random walks,
\newblock {\it Ann. Prob.} {\bf 25} 531-572 (1997). 
 \smallskip
 
 
 \bibitem{DV}
{\sc M.D.~Donsker} and {\sc S.R.S.~Varadhan},
\newblock Asymptotic evaluation of certain Markov process expectations for large time, I--IV,
\newblock {\it Comm.~Pure Appl.~Math.} {\bf 28}, 1--47, 279--301 (1975), {\bf 29}, 389--461 (1979), {\bf 36}, 183--212 (1983)

\smallskip

\bibitem{DV75}
{\sc M.D.~Donsker} and {\sc S.R.S.~Varadhan},
\newblock Asymptotics of the Wiener sausage,
\newblock {\it Comm.~Pure Appl.~Math.} XXVIII 525-565 (1975).

\bibitem{DVP}
{\sc M.D.~Donsker} and {\sc S.R.S.~Varadhan},
\newblock Asymptotics for the Polaron,
\newblock {\it Comm.~Pure Appl.~Math.} 505-528 (1983).


\bibitem{L}
{\sc E.H.~Lieb},
\newblock Existence and uniqueness of the minimizing solution of Choquard�s nonlinear equation,
\newblock{\it Studies in Appl. Math.} {\bf 57}, 93-105 (1976)

\bibitem{P}
{\sc N.I.~Portenko},
\newblock{ Diffusion processes with unbounded drift coefficient,}
\newblock{\it  Theoret. Probability Appl.}  {\bf 20}, pp. 27-31 (1976)
 \end{thebibliography}
\end{document}